\documentclass[a4paper,12pt]{amsart}
\usepackage{amsfonts}
\usepackage{a4wide}
\usepackage{amsthm}
\usepackage{amsmath}
\usepackage{amssymb}
\usepackage[final]{graphicx}
\usepackage{graphics}
\usepackage{graphicx}
\usepackage{epsfig}
\usepackage{tabularx}
\usepackage{enumerate}
\usepackage[all]{xy}
\usepackage{eepic}
\usepackage[UKenglish]{babel}
\usepackage{verbatim}

\newtheorem{theorem}{Theorem}[section]
\newtheorem{corollary}[theorem]{Corollary}

\newtheorem{lemma}[theorem]{Lemma}
\newtheorem{proposition}[theorem]{Proposition}

\theoremstyle{definition}
\newtheorem{definition}[theorem]{Definition}

\title{Intrinsic Ergodicity of Open dynamical systems for the doubling map.}
\author{Rafael Alcaraz Barrera}
\date{\today}
\address{Departamento de Matem\'atica aplicada\\
IME-USP\\
Rua do Mat\~ao 1010, Ci\-da\-de U\-ni\-ver\-si\-ta\-ria, 05508-090, S\~ao Paulo SP, Brazil}
\email{rafalba@ime.usp.br} 

\begin{document}
\begin{abstract}
We give sufficient conditions for intervals $(a,b)$ such that the associated open dynamical system for the doubling map is intrinsically ergodic. We also show that the set of parameters $(a,b) \in (\frac{1}{4}, \frac{1}{2}) \times (\frac{1}{2},\frac{3}{4})$ such that the attractor $(\Lambda_{(a,b)}, f_{(a,b)})$ is intrinsically ergodic has full Lebesgue measure and we construct a set of points where intrinsic ergodicity does not hold. This paper continues the work started in \cite{yomero3}.   
\end{abstract}
\maketitle

\section{Introduction and Summary}
\label{intro}

\noindent Since their introduction by Pianigianni and Yorke in 1979 \cite{yorkep}, the study of \textit{open dynamical systems} (colloquially, \textit{maps with holes}) has become very active in recent years -see e.g. \cite{markarian1, dysman, fisher, viana1, viana2, misiurewicz} among others. Let us remind the reader the general setting. Given a discrete dynamical system $(X,f)$, where $X$ is a compact metric space and $f:X \to X$ is a continuous and surjective transformation with positive topological entropy, and an open set $U \subset X$, we define \textit{the survivor set corresponding to $U$} as $$X_U = \left\{x \in X : f^n(x) \notin U \hbox{\rm{ for every }} n \in \mathbb{N} \right\}.$$ Since $X_U$ is a forward invariant set, it is possible to consider the dynamical system $(X_U,f_U)$ where $f_U = f\mid_U$ and $(X_U,f_U)$ is called \textit{the open dynamical system corresponding to $U$}. 

\vspace{1em}There are some interesting questions regarding the topological dynamics of $(X_U,f_U)$ as well as the ergodic properties of the invariant measures supported on $X_U$. In particular, it is an intriguing question to determine when an open dynamical system $(X_U, f_U)$ is \textit{intrinsically ergodic}, which is a central problem relating ergodic theory and topological dynamics \cite{climenhaga}. 

\vspace{1em}The purpose of this paper is to study intrinsic ergodicity for some families of open dynamical systems corresponding to the doubling map $2x \mod 1$. Understanding the dynamical properties of open dynamical systems of the doubling map as well as studying the fine properties $X_{(a,b)}$ has awaken interest recently - see \cite{dettman, sidorov1, nilsson, sidorov2} among others. We are interested in understand the following three situations; firstly we want to understand intrinsic ergodicity for holes $(a,b) \subset S^1$ when $b =1$ or $a=0$, secondly when $a= \frac{1}{2}$ or $b = \frac{1}{2}$ and finally we want to determine when the dynamical system given by $(\Lambda_{(a,b)}, f_{(a,b)})$ where $\Lambda_{(a,b)} = S^1_{(a,b)} \cap \left[2b-1, 2a\right]$ is intrinsically ergodic whenever $(a,b)$ is a \textit{centred hole} i.e $\frac{1}{2} \in (a,b)$ and $\Lambda_{(a,b)}$ has positive Hausdorff dimension. This question was posed by Sidorov in \cite{sidorov4}. Our strategy is to answer it is to use the symbolic properties of the binary expansions of the boundary points of the interval as suggested in \cite{bundfuss} as well as tools from symbolic dynamics. 
 
\vspace{1em}The structure of the paper is as follows: In Section \ref{basic} we mention all the tools from symbolic dynamics and ergodic theory used during the paper. In Section \ref{laprevia} we give a brief exposition of research previously undertaken on the intrinsic ergodicity of families open dynamical systems of the doubling map. In particular we explain why the results obtained by Bundfuss et. al \cite{bundfuss}, Nilsson \cite{nilsson}, Climenhaga and Thompson \cite{climenhaga} will imply the intrinsic ergodicity of open dynamical systems corresponding to holes of the form $(0,b)$ and $(a,1)$ whenever $b < \frac{1}{2}$ or $a > \frac{1}{2}$. Also, we show that $\left(\Lambda_{(a,b)}, f_{(a,b)}\right)$ is intrinsically ergodic when $a < \frac{1}{2}$ and  $(a,b) = \left(a, \frac{1}{2}\right)$. 

\vspace{1em}In Section \ref{intrinsicergodicity} we show our main theorem: 

\newtheorem*{elmerobueno}{Theorem \ref{elmerochingon}}
\begin{elmerobueno}
The set $$D_{I} = \left\{(a,b) \in D_1 : (\Lambda_{(a,b)},f_{(a,b)}) \hbox{\rm{ is intrinsically ergodic }} \right\}$$ has full Lebesgue measure, where $$D_1 = \left\{(a,b) \in \left(\frac{1}{4},\frac{1}{2}\right) \times \left(\frac{1}{2},\frac{3}{4}\right): \dim_{H}(X_{(a,b)}) > 0\right\}.$$
\end{elmerobueno}

The main idea to prove Theorem \ref{elmerochingon} is to describe the topological entropy of the \textit{transitive components of $\left(\Sigma_{(\alpha,\beta)}, \sigma_{(\alpha,\beta)}\right)$}, where $\left(\Sigma_{(\alpha,\beta)}, \sigma_{(\alpha,\beta)}\right)$ is the \textit{lexicographic subshift} associated to $\left(\Lambda_{(a,b)}, f_{(a,b)}\right)$ - see \cite[Corollary 3.21]{yomero3}- in order to show the uniqueness of the measure of maximal entropy of $\left(\Lambda_{(a,b)}, f_{(a,b)}\right)$. To the best of the knowledge of the author, the notion of transitive component for an open dynamical system was introduced by Bundfuss, et.al. in \cite{bundfuss}. Also, we use the connection between open dynamical systems for the doubling map and Lorenz maps established in \cite{yomero3}, the entropy formula introduced by Glendinning and Hall in \cite{hall} and a suitable division of $D_1$ given by \textit{$n$-renormalisation boxes} - see Definitions \ref{0renormalisationbox} and \ref{nrenormalisationbox} to describe such measure for every $(a,b) \in D_1$. Renormalisation boxes are constructed via essential pair i.e. parameters $(\alpha, \beta)$corresponding to transitive subshifts of finite type and substitution systems depending on balanced words associated to the elements of $\mathbb{Q}\cap(0,1)$ - see Section \ref{basic}. Using substitution systems is a similar approach to the one used by Glendinning and Sidorov in \cite{sidorov1} to describe the set $D_1$.  

\vspace{1em}We also mention that there exists a countable family of non transitive attractors due to Hofbauer \cite{hofbauer1} and Glendinning and Hall \cite{hall} showing in Theorem \ref{canon1} that none of the elements of the mentioned family are intrinsically ergodic. Using this family of examples and substitutions systems again, we construct countably many attractors where intrinsic ergodicity does not hold in Theorem \ref{laconstruccion}. 

\section[Preliminaries]{Background}
\label{basic}

\noindent For the convenience of the reader, we give all the relevant concepts from Symbolic Dynamics and Ergodic Theory to develop our study.

\subsection*{Symbolic Dynamics}

\noindent For a full exposition in symbolic dynamics the reader can consult \cite{lindmarcus}. We will consider subshifts defined on the alphabet $\{0,1\}$ only and we will call its elements \textit{symbols} or \textit{digits}. Let $\Sigma_2 = \mathop{\prod}\limits_{n=1}^{\infty}\{0, 1\}$, i.e. $\Sigma_2$ is the set of all one sided sequences with symbols in $\{0,1\}$. Recall that $\Sigma_2$ is a compact metric space with the distance given by: $$d(x,y) = \left\{
\begin{array}{clrr}      
2^{-j} & \hbox{\rm{ if }}  x \neq y;& 
\hbox{\rm{where }} j = \min\{i : x_i \neq y_i \}\\
0 & \hbox{\rm{ otherwise.}}&\\
\end{array}
\right.$$ 

Recall that $\pi: \Sigma_2 \to \left[0,1\right]$ is the \textit{the projection map} given by $\pi(x) = \mathop{\sum}\limits_{i=1}^{\infty} \frac{x_i}{2^i}$. Let $\sigma: \Sigma_2 \to \Sigma_2$ given by $\sigma((x_i)_{i=1}^{\infty}) = (x_{i+1})_{i=1}^{\infty}$ \textit{the one-sided full shift map}. Let $A \subset \Sigma_2$. We say that $(A, \sigma_A)$ is a \textit{subshift of $\Sigma_2$} or simply \textit{a subshift} if $A$ is a closed and $\sigma$-invariant set, and $\sigma_A$ is defined by $\sigma_A = \sigma \mid_A$.  

\vspace{1em} A finite sequence of symbols $\omega = w_1,\ldots w_n$ with $w_i \in \{0,1\}$ will be referred as a word and given a word $\omega$ we denote the \textit{length of $\omega$} by $\ell(\omega)$. Note that given two finite words $\omega = w_1,\ldots w_n$ and $\nu = u_1 \ldots u_m$ their \textit{concatenation} denoted by $\omega\nu$ is given by $\omega\nu = w_1,\ldots w_n u_1 \ldots u_m$. We denote by $\omega^n$ the word $\omega$ concatenated to itself $n$ times. Given a sequence $x \in \Sigma_2$ and a word $\omega$ we say that \textit{$\omega$ is a factor of $x$} or \textit{$\omega$ occurs in $x$} if there are coordinates $i$ and $j$ such that $\omega = x_i \ldots x_j$. Note that the same definition holds if $x$ is a finite word. 

\vspace{1em}It is necessary for our exposition to use some tools from combinatorics of words -see \cite[Chapter 2]{lothaire} for a detailed exposition. For every finite word $\omega$ we denote by $0-\max_{\omega}$ to the lexicographically largest cyclic permutation of $\omega$ starting with $0$ and by $1-\min_{\omega}$ to the lexicographically smallest cyclic permutation starting with $1$. We denote by $\left|\omega \right|_1$ to the number of $1$'s of a finite word $\omega$ and define \textit{the $1$-ratio of $\omega$} to be $1(\omega) = \frac{\left|\omega\right|_1}{\ell(\omega)}$. A word $\omega$ is said to be \textit{balanced} if for any pair of factors of $\omega$, $\nu$ and $\upsilon$ of length $n$ with $2 \leq n \leq \ell(\omega)$, $\left|\left|\upsilon\right|_1 - \left|\nu\right|_1 \right| \leq 1,$ and we say that $\omega$ is \textit{cyclically balanced} if $\omega^2$ is balanced. Note that given a $r = \frac{p}{q} \in \mathbb{Q}\cap(0,1)$ there exist only $q$ distinct cyclically balanced words with 
length $q$ and $p$ $1$'s. For $r = \frac{p}{q} \in \mathbb{Q}\cap(0,1)$, $\xi_r$ stands for the lexicographically largest cyclically balanced word starting with $0$ with $\ell(\xi_r) =$ and $1(\xi_r) = r$. Also $\zeta_r$ stands for the lexicographically smallest cyclically balanced word starting with $1$ with $\ell(\zeta_r) = p$ and $1(\zeta_r) = r$. Note that $\xi_r = 0-\max_{\xi_r}$ and $\zeta_r = 1-\min_{\zeta_r}$ and that $\xi_r$ and $\zeta_r$ are cyclic permutations of each other.

\vspace{1em}Given two words $\omega$ and $\nu$ we define $\rho_{\omega,\nu}$ to be the substitution given by $\rho_{\omega, \nu}(0) = \omega$ and $\rho_{\omega,\nu}(1)$. Given a word $\upsilon$, $\rho_{\omega, \nu}(\upsilon) = \rho_{\omega, \nu}(u_1) \ldots \rho_{\omega, \nu}(u_{\ell(\upsilon)})$. If $\omega = \xi_r$ and $\nu = \zeta_r$ for $r \in \mathbb{Q}\cap(0,1)$, the substitution $\rho_{\xi_r, \zeta_r}$ will be denoted simply as $\rho_r$. 

\vspace{1em}Let $\mathcal{F}$ be a set of words and let $$\Sigma_{\mathcal{F}} = \left\{x \in \Sigma_2 : \upsilon \hbox{\rm{ is not a factor of }} x \hbox{\rm{ for any word }} \upsilon \in \mathcal{F}\right\}.$$ Since $\Sigma_{\mathcal{F}}$ is a closed and $\sigma$-invariant set, the dynamical system given by $(\Sigma_{\mathcal{F}}, \sigma\mid_{\Sigma_{\mathcal{F}}})$ is a subshift of $\Sigma_2$. Conversely, for every compact and $\sigma$-invariant set $A$, there always exist a set of forbidden factors $\mathcal{F}$ such that $A = \Sigma_{\mathcal{F}}$ \cite[Theorem 6.1.21]{bundfuss}. 

\vspace{1em}We say that a subshift $(\Sigma_{\mathcal{F}},\sigma\mid_{\Sigma_{\mathcal{F}}})$ is a \textit{subshift of finite type} if $\mathcal{F}$ is finite. We say that a subshift $(\Sigma_{\mathcal{F}},\sigma\mid_{\Sigma_{\mathcal{F}}})$ is \textit{sofic} if there is a subshift of finite type $(A, \sigma_{A})$ and a semi-conjugacy $h:A \to \Sigma_{\mathcal{F}}$. Given to finite words $\omega$, $\nu$ let $\left\{\omega, \nu\right\}^{\infty}$ denote the set of all concatenations of $\omega$ and $\nu$ together with their shifts. We call the dynamical system $\left(\left\{\omega, \nu\right\}^{\infty}, \sigma\mid_{\left\{\omega, \nu\right\}^{\infty}}\right)$ a \textit{one-sided uniquely decipherable renewal system} - see \cite{lind} for a general definition. 

\vspace{1em}Let $\alpha = (a_i)_{i=1}^{\infty}, \beta = (b_i)_{i=1}^{\infty} \in \Sigma_2$. We say that $\alpha$ is \textit{lexicographically less than} $\beta$, denoted by $\alpha \prec \beta$ if there exists $k \in \mathbb{N}$ such that $a_j = b_j$ for $j < k$ and $a_k < b_k$. We denote by $\left(\alpha, \beta\right)$ to the \textit{lexicographic open interval from $\alpha$ to $\beta$}, i.e. $\left(\alpha,\beta \right) = \left\{x \in \Sigma_2 : \alpha \prec x \prec \beta \right\}.$ Similarly, it is possible to consider \textit{lexicographic closed intervals} by replacing $\prec$ for $\preccurlyeq$. Given a sequence $x \in \Sigma_2$, \textit{the mirror image of $x$} is the sequence $\bar{x} = (1 - x_i)_{i=1}^{\infty}$. 

\vspace{1em}In particular, we are interested in the following family of subshifts. Let $(\alpha,\beta)\in \Sigma_2 \times \Sigma_2$. We define the \textit{lexicographic subshift corresponding to $(\alpha, \beta)$} by considering $$\Sigma_{(\alpha,\beta} = \left\{x \in \Sigma_2 : \beta \preccurlyeq \sigma^{n}(x) \preccurlyeq \alpha \hbox{\rm{ for every }} n \geq 0\right\}$$ and $\sigma_{(\alpha, \beta)} = \sigma\mid_{\Sigma_{(\alpha,\beta)}}$.

\vspace{1em}A sequence $\alpha \in \Sigma_2$ is said to be a \textit{Parry sequence} if $a_1 =1$ and $\sigma^n(\alpha) \preccurlyeq \alpha$ for every $n \in \mathbb{N}$. We denote the set of Parry sequences by $P$ and by $\bar {P} = \left\{x \in \Sigma_2 : \bar{x} \in P\right\}$. 

\begin{definition}
A pair $(\alpha, \beta)$ such that $\alpha \in P$ and $\beta \in \bar{P}$, and 
$$\sigma^{n}(\alpha) \succcurlyeq \beta \hbox{\rm{ and }} \sigma^{n}(\beta) \preccurlyeq \alpha \hbox{\rm{ for every }} n \in \mathbb{N}$$ is called an \textit{admissible pair}. Otherwise the pair $(\alpha,\beta)$ is said to be \textit{extremal}. The family of admissible is called \textit{the lexicographic world} and we will denote it as $\mathcal{LW}$ \label{lexworld}.
\end{definition}

\vspace{1em}As a consequence of \cite[Theorem 3.6]{nilsson}, $\pi(\mathcal{LW})$ has Lebesgue measure zero. Also, observe that if $(\alpha, \beta) \in \mathcal{LW}$ then neither $\alpha$ nor $\beta$ have arbitrarily long strings of $0$'s or $1$'s unless $\alpha = 1^{\infty}$ or $\beta = 0^{\infty}$ and $\Sigma_{(\alpha,\beta)} = \Sigma_2$ if and only if $\alpha = 1^{\infty}$ and $\beta = 0^{\infty}$. 

\vspace{1em}Given a lexicographic subshift $(\Sigma_{(\alpha,\beta)},\sigma_{(\alpha,\beta)})$, \textit{the set of admissible words of length $n$ of $(\Sigma_{(\alpha,\beta)}, \sigma_{(\alpha,\beta)})$} is denoted by $B_n(\Sigma_{(\alpha,\beta)})$ and the \textit{language of $\Sigma_{(\alpha, \beta)}$} is $$\mathcal{L}(\Sigma_{(\alpha,\beta)}) = \mathop{\bigcup}\limits_{n=1}^{\infty}B_n(\Sigma_{(\alpha,\beta)}).$$ We define the \textit{topological entropy of $(\Sigma_{(\alpha,\beta)}, \sigma_{(\alpha,\beta)})$} by $$h_{top}(\sigma_{(\alpha,\beta)}) = \mathop{\lim}\limits_{n \to \infty} \dfrac{1}{n}\log \left|B_n(\Sigma_{(\alpha,\beta)})\right|$$ where $\log$ is always considered to be $\log_2$. Given $(\alpha,\beta) \in \mathcal{LW}$, \textit{the entropy formula for $(\Sigma_{(\alpha,\beta)}, \sigma_{(\alpha,\beta)})$} by $$K_{(\alpha,\beta)}(t) = \mathop{\sum}\limits_{i=0}^{\infty}(b_i-a_i)t^{i}$$ where $a_0 = 0$ and $b_0 = 1$.  

\vspace{1em}The entropy formula was introduced in Glendinning and Hall \cite{hall} in the context of for Lorenz maps. Recently, Barnsley, Steiner and Vince in \cite{barnsley} give a symbolic proof showing that the smallest positive root of $K(t)$, denoted by $\kappa$, satisfies $$\kappa = \left(\mathop{\lim}\limits_{n\to \infty}\sqrt[n]{\lvert B_n(\Sigma_{(\alpha,\beta)})} \rvert \right)^{-1}.$$ Thus, $\log (\frac{1}{\kappa}) = h_{top}(\sigma_{(\alpha,\beta)})$. Note that given two words $\omega$ and $\nu$ we can calculate the topological entropy of $\left(\left\{\omega,\nu\right\}^{\infty}, \sigma \mid_{\left\{\omega,\nu\right\}^{\infty}}\right)$ in a similar way, namely $h_{top}(\sigma \mid_{\left\{\omega,\nu\right\}^{\infty}}) = \log(\frac{1}{\lambda})$ where $\lambda$ is the unique root of $1-t^{\ell(\omega)}-t^{\ell(\nu)}$ in $[0,1]$ \cite[Remark 2.1]{hong} \cite[p. 1008]{hall}. 

\vspace{1em}We say that lexicographic subshift $(\Sigma_{(\alpha,\beta)}, \sigma_{(\alpha,\beta)})$  
\begin{enumerate}[$i)$]
\item is \textit{topologically transitive} if for any two ordered words $\omega, \nu \in \mathcal{L}(\Sigma_{(\alpha,\beta)})$ there exist a word $\upsilon \in \mathcal{L}(\Sigma_{(\alpha,\beta)})$ (which we refer as \textit{bridge}) such that $\omega \upsilon \nu \in \mathcal{L}(\Sigma_{(\alpha,\beta)})$; 
\item is \textit{topologically mixing} if for every ordered pair of words $\upsilon, \nu \in \mathcal{L}(\Sigma_{(\alpha,\beta)})$, there exists $N \in \mathbb{N}$ such that for each $n \geq N$ there is a bridge $\omega \in B_n(\Sigma_{(\alpha,\beta)})$, such that $\upsilon \omega \nu \in \mathcal{L}(\Sigma_{(\alpha,\beta)})$;  
\item has \textit{the specification property} if $(\Sigma_{(\alpha,\beta)}, \sigma_{(\alpha,\beta)})$ is transitive and there exist $M \in \mathbb{N}$ such that for every $\omega, \nu$, $\ell(\upsilon) = M$;  
\item is \textit{coded} if there exist a sequence of transitive lexicographic subshifts of finite type $\left\{(\Sigma_{(\alpha_n,\beta_n)},\sigma_{(\alpha_n,\beta_n)})\right\}_{n=1}^{\infty}$ such that $\Sigma_{(\alpha_n,\beta_n)} \subset \Sigma_{(\alpha_{n+1},\beta_{n+1})}$ for every $n \in \mathbb{N}$ and $$\Sigma_{(\alpha,\beta)} = \overline{\mathop{\bigcup}\limits_{n=1}^{\infty} \Sigma_{(\alpha_n,\beta_n)}}.$$
\end{enumerate}

Note that one-sided uniquely decipherable renewal systems are transitive sofic subshifts.

\subsection*{Ergodic Theory} For a comprehensive presentation on Ergodic Theory we refer the reader to \cite{mane, walters}. Given a lexicographic subshift $(\Sigma_{(\alpha,\beta)}, \sigma_{(\alpha,\beta)})$, $\mathbb{P}(\Sigma_{(\alpha,\beta)})$ stands for the set of all probability measures defined on the Borel $\sigma$-algebra, $\mathcal{B}(\Sigma_{(\alpha,\beta)})$. A measure $\mu \in \mathbb{P}(\Sigma_{(\alpha,\beta)})$ is \textit{$\sigma_{(\alpha,\beta)}$-invariant} if for every $A \in \mathcal{B}(\Sigma_{(\alpha,\beta)})$, $\mu({\sigma_{(\alpha,\beta)}}^{-1}(A)) = \mu(A)$. We denote by $\mathbb{M}(\sigma_{(\alpha,\beta)})$ to the set of $\sigma_{(\alpha,\beta)}$-invariant measures on $\Sigma_{(\alpha,\beta)}$. 

\vspace{1em}Given a countable partition $\boldsymbol{\alpha}$ of $\Sigma_{(\alpha,\beta)}$ and $\mu \in \mathbb{M}(\sigma_{(\alpha,\beta)})$, the \textit{entropy of $\mu$ relative to the partition $\boldsymbol{\alpha}$} is $$h_{\mu}(\sigma_{(\alpha,\beta)}, \boldsymbol{\alpha}) = \mathop{\lim}\limits_{n \to \infty} \dfrac{1}{n} H\left(\mathop{\bigvee}\limits_{j=0}^{n-1}{\sigma_{(\alpha,\beta)}}^{-j}(\boldsymbol{\alpha})\right),$$ where $$H(\boldsymbol{\alpha}) = -\mathop{\sum}\limits_{A \in \boldsymbol{\alpha}}\mu(A)\log \mu(A),$$ and \textit{the measure theoretical entropy of $\sigma_{(\alpha,\beta)}$ with respect to $\mu$} or simply  \textit{the entropy of $\mu$} is $$h_{\mu}(\sigma_{(\alpha,\beta)}) = \sup\{h_{\mu}(\sigma_{(\alpha,\beta)},\boldsymbol{\alpha}) : \boldsymbol{\alpha} \hbox{\rm{ is a countable partition of }}\Sigma_{(\alpha,\beta)}\}.$$

\vspace{1em}There is a relationship between the topological entropy of a lexicographic subshift  $(\Sigma_{(\alpha,\beta)},\sigma_{(\alpha,\beta)})$ and the measure theoretical entropy of $\mu \in \mathbb{M}(\sigma_{(\alpha,\beta)})$ known as \textit{the variational principle} \cite[Theorem 8.6]{walters}, that is $$h_{top}(\sigma_{(\alpha,\beta)}) = \sup \{h_{\mu}(\sigma_{(\alpha,\beta)}) : \mu \in \mathbb{M}(\sigma_{(\alpha,\beta)})\}.$$ A $\sigma_{(\alpha,\beta)}$-invariant measure $\mu$ satisfying the variational principle is called \textit{a measure of maximal entropy}. We say that a lexicographic subshift $(\Sigma_{(\alpha,\beta)},\sigma_{(\alpha, \beta)})$ is \textit{intrinsically ergodic} if there is a unique measure of maximal entropy. Note that if $h_{top}(\sigma_{(\alpha, \beta)}) = 0$ it will not be intrinsically ergodic unless $(\Sigma_{(\alpha, \beta)},\sigma_{(\alpha, \beta)})$ is uniquely ergodic since $\mathbb{M}(\sigma_{(\alpha, \beta)}) = {\mathbb{M}}_{\max}(\sigma_{(\alpha, \beta)})$ where 
${\mathbb{M}}_{\max}(\sigma_{(\alpha, \beta)})$ is the set of measures of maximal entropy for $(\Sigma_{(\alpha, \beta)}, \sigma_{(\alpha, \beta)})$ \cite{walters}.

\vspace{1em}It is known that transitive (mixing) subshifts of finite type are intrinsically ergodic \cite{parry} as well as transitive sofic subshifts \cite{weiss1, weiss2} and subshifts with the specification property \cite{bowen1}. Recently, Climenhaga and Thompson in \cite{climenhaga} shown a criteria to determine when a coded system is intrinsically ergodic. Finally, Gurevi{\v{c}} in \cite{gurevich} gave a general criteria to determine the intrinsic ergodicity of a subshift.  It is worth mentioning that intrinsic ergodicity will not follow neither from the topological transitivity nor topological mixing of a subshift - see, e.g. \cite{gurevich, petersen}.

\section[$\beta$-expansions]{Open dynamical systems for non centred holes, $\beta$-expansions and intrinsic ergodicity}
\label{laprevia}

\noindent Consider $S^1 = \mathbb{R} / \mathbb{Z}$ and let $f:S^1 \to S^1$ be the doubling map $2x \mod 1$. Recall that given an interval $(a,b) \subset S^1$ the \textit{$(a,b)$-exceptional set} is given $$X_{(a,b)} = \left\{x \in S^1: f^n(x) \notin (a,b) \hbox{\rm{ for every }} n \geq 0\right\}$$ and $(X_{(a,b)}, f_{(a,b)})$ is the open dynamical system corresponding to $(a,b)$ where $f_{(a,b)} = f\mid_{X(a,b)}$. We say that an interval $(a,b) \subset S^1$ is a \textit{centred hole} if $\frac{1}{2} \in (a,b)$. Given a centred hole $(a,b)$, the \textit{attractor of $X_{a,b}$} is $\Lambda_{(a,b)} = [2b-1,2a]\cap X_{(a,b)}$.

\vspace{1em}Assume that $(a,b) \subset S^1$ satisfies that $a \in (\frac{1}{2},1)$and $b=1$. Then $$X_{(a,b)} = \left\{x \in S^1 : f^n(x) \leq a \hbox{\rm{ for every }} n \geq 0 \right\}.$$

\vspace{1em}Let $\beta \in (1,2)$ and $x \in \left[0,1\right]$. We say that a sequence $\{b_n\}_{n=1}^{\infty} \in \Sigma_2$ is \textit{a $\beta$-expansion for $x$} if $\{b_n\}_{n=1}^{\infty}$ satisfies $$x = \mathop{\sum}\limits_{n=1}^{\infty}b_n \beta^{-n}.$$ $\beta$-expansions were introduced Parry and R\'enyi in \cite{parry2, renyi} as a generalisation of the expansions with integer basis. To find one of such expansions for $x \in [0,1]$ an algorithm is provided by the \textit{$\beta$-transformation}. Given $\beta \in (1,2)$, the \textit{$\beta$-trans\-for\-ma\-tion} is the transformation $\tau_{\beta}:[0,1] \to [0,1]$ given by $\tau_{\beta}(x) = \beta x \mod 1$. Note that if $\beta = 2$ then  $\tau_{\beta} = 2x \mod 1$. Using $\tau_{\beta}$, we construct a $\beta$-expansion for $x \in [0,1)$ by $b_n = [\beta \tau_{\beta}^{n-1}(x)]$ for $n \in \mathbb{N}$. The obtained $\beta$-expansion for $x$ is known as \textit{the greedy expansion of $x$}.

\vspace{1em}Let $X_{\beta} \subset \Sigma_2$ be the set of all greedy expansions corresponding to $\beta$. Then, $\sigma_{\beta}: X_{\beta} \to X_{\beta}$ by $\sigma_{\beta} = \sigma\mid_{X_{\beta}}$is a subshift. We shall call the subshift $(X_{\beta}, \sigma_{\beta})$ \textit{the usual $\beta$-shift}. The properties of the usual $\beta$-shift have been extensively studied. In particular, for every $\beta \in (1,\infty)$, the usual $\beta$-shift is a topologically mixing subshift and $h_{top}(\sigma_{\beta}) = \log (\beta)$ \cite{sidorov4}. It is shown in \cite{sidorov4} that $(X_{\beta}, \sigma_{\beta})$ is topologically conjugated to $([0,1], \tau_{\beta})$. Moreover, Parry in \cite{parry2} showed that $$X_{\beta} = \left\{x \in \Sigma_2 : \sigma^n(x) \prec 1_{\beta} \hbox{\rm{ for every }} n \in \mathbb{N}\right\},$$ where $1_{\beta} = \{d_i\}_{i=1}^{\infty}$ is the greedy expansion of $1$ if $0^{\infty}$ is not a factor of $1_{\beta}$ is not a finite sequence and $1_{\beta} = (d_1,\ldots d_{k-1}0)^{\infty}$ if $0^{\infty}$ is a factor of $\{d_i\}_{i=1}^{\infty}$ and $k$ satisfies that for every $i > k$ $d_j = 0$. The described expansion is called \textit{quasi-greedy expansion of $1$}. Moreover, Parry in \cite{parry2} characterised lexicographically the set of sequences that are greedy $\beta$-expansions of $1$ is, namely, $\{d_i\}_{i=1}^{\infty}$ is a greedy expansion of $1$ if and only if $\{d_i\}_{i=1}^{\infty}$ is a Parry sequence. 

\vspace{1em}Note that $(X_{\beta}, \sigma_{\beta})$ is a subshift of finite type if and only $1_{\beta}$ is periodic or finite and $(X_{\beta}, \sigma_{\beta})$ is sofic if and only $1_{\beta}$ is preperiodic \cite[Theorem 2.2]{sidorov4}. In addition, Bertrand-Mathis in \cite{bertrandmathis} shown that $(X_{\beta}, \sigma_{\beta})$ has the specification property if and only $1_{\beta}$ does not contain blocks of consecutive $0$'s of arbitrary length. Finally, it is known that the usual $\beta$-shift is intrinsically ergodic for every $\beta$ \cite{hofbauer}. 

\vspace{1em}Assume that $\alpha$ satisfies that $a_1 = 1$ and $\alpha \notin P$. Then we define $$\varsigma(\alpha) = (a_1 \ldots a_{n_{\alpha-1}}0)^{\infty}$$ where $$n_{\alpha} = \min\left\{n \in \mathbb{N} : \sigma^n(\alpha) \succcurlyeq \alpha \right\}$$ - see \cite[p. 11]{yomero3}. Note that $\varsigma(\alpha) \in P$. From \cite[Lemma 3.12]{yomero3} we are sure that $\Sigma_{(0^{\infty}, \alpha)} = \Sigma_{(0^{\infty}, \varsigma(\alpha))}.$ Also, it is clear that there is $\beta = \beta(a) \in \Sigma_{2}$ satisfying that $\Sigma_{(0^{\infty}, \alpha)} = X_{\beta(a)}$. In \cite[Theorem 3.5]{bundfuss}, Bundfuss et. al. showed that $(X_{\beta}, \sigma_{\beta})$ is topologically conjugated to $(X_{(a,1)}, f_{(a,1)})$ with $a \in (\frac{1}{2}, 1)$.  Also, Nilsson in \cite{nilsson} proved the same result for open dynamical systems of the form $(X_{(0,b)}, f_{(0,b)})$ with $b \in (0, \frac{1}{2})$. As a consequence of \cite[Theorem A]{climenhaga} we have the following proposition: 

\begin{proposition}
The open dynamical systems $(X_{(a,1)}, f_{(a,1)})$ and $(X_{(0,b)}, f_{(0,b)})$ are intrinsically ergodic provided that $a \in \left(\frac{1}{2}, 1\right)$ and $b \in \left(0, \frac{1}{2}\right)$. \label{proposicionchafa} 
\end{proposition}

Now we consider another class of non-centred holes, namely intervals of the form $\left(a, \frac{1}{2}\right)$. Observe that if $a \leq \frac{1}{2}$ then by \cite[Lemma 1.1]{sidorov1} we are sure that $X_{(a, \frac{1}{2})} = \left\{0\right\}$, then we will restrict ourselves to $a \in (\frac{1}{4}, \frac{1}{2})$. 

Let $a \in (\frac{1}{4}, \frac{1}{2})$ and fix the binary expansion of $\frac{1}{2}$ to be $10^{\infty}$. Observe that the attractor of $(X_{(a,\frac{1}{2})}, f_{(a, \frac{1}{2}})$, $\Lambda_{(a, \frac{1}{2})} = X_{(a, \frac{1}{2})} \cap [0, 2a]$. Let $\alpha = \pi^{-1}(2a)$ and consider $\Sigma_{(\varsigma(\alpha),0^{\infty})}$. Recall that if $\alpha \in P$ then $\Sigma_{(\alpha,0^{\infty})}$ coincides with the greedy $\beta(\alpha)$-shift up to a countable set of sequences given by $$\left(\mathop{\bigcup}\limits_{n=0}^{\infty} \sigma^{-n}(0) \right) \cap \Sigma_{(\alpha, 0^{\infty})}$$ and $\beta(\alpha)$ is determined by the unique positive solution to the equation $$1 = \mathop{\sum}\limits_{n=1}^{\infty}\dfrac{a_i}{\beta(\alpha)^i}.$$

\begin{lemma}
For every $a \in (\frac{1}{4}, \frac{1}{2})$, there is $\beta \in (1,2)$ such that the dynamical systems $(\Lambda_{(a, \frac{1}{2})}, f_{(a, \frac{1}{2})})$ and $(\Sigma_{\beta}, \sigma_{\beta})$ are topologically conjugated.\label{lemabeta}
\end{lemma} 
\begin{proof}
Note that if $\alpha \in P$ then $\pi^{-1}(\Lambda_{(a,\frac{1}{2})}) =  X_{\beta(\alpha)}$. Therefore our result holds. 

Consider $a \in (\frac{1}{4}, \frac{1}{2})$ such that $\alpha \notin P$. Observe that $\varsigma(\alpha) \prec \pi^{-1}(2a)$, then $X_{\beta(\varsigma(\alpha)}) \subset \pi^{-1}(\Lambda_{(a,\frac{1}{2})})$. Assume that there exists $x \in \pi^{-1}(\Lambda_{(a,\frac{1}{2})} \setminus X_{\beta(\varsigma(\alpha)})$. This implies that $x_i = \varsigma(\alpha)_i$ at least for every $1 \leq i \leq n_{\alpha}$. Let $k = \min \{j \geq n_{\alpha} : \varsigma(\alpha)_j < x_j\}$. Observe that $x_k \leq a_j$. Since $x \in \pi^{-1}(X_{(a,\frac{1}{2})})$ then $\sigma^n(x) \prec \alpha$ for every $n \geq 0$. Consider $\sigma^{n_{\alpha}}(x)$. Note that $\sigma^{n_{\alpha}}(x)_{j-n_{\alpha}} \succ \varsigma(\alpha)_{j-n_{\alpha}}$. This implies that $\sigma^{n_{\alpha}}(x)_{j-n_{\alpha}} > \alpha_{j-n_a}$, that is $x \notin \pi^{-1}(\Lambda_{(a,\frac{1}{2})})$, which is a contradiction. Therefore $\pi^{-1}(\Lambda_{(a,\frac{1}{2})}) = X_{\beta(\varsigma(\alpha))}$. Then $\pi^{-1} \circ \pi_{\beta(\varsigma(\alpha))}$ is a topological conjugation between $(\Lambda_{\left(a,\frac{1}{2}\right)}, f_{(a,\frac{1}{2})})$ and $\left(X_{\beta(\varsigma(\alpha))}, \sigma_{\beta(\varsigma(\alpha))} \right)$. 
\end{proof}

Observe that Lemma \ref{lemabeta} combined with \cite[Theorem A]{climenhaga} gives us that $(\Lambda_{(a, \frac{1}{2})}, f_{(a, \frac{1}{2})})$ is intrinsically ergodic for every $a \in (\frac{1}{4}, \frac{1}{2})$.

\vspace{1em}To the best of our knowledge, determining when $(X_{(a,b)},f_{(a,b)})$ or $(\Lambda_{(a,b)},f_{(a,b)})$ are intrinsically ergodic provided that either $(a,b) \subset \left(0,\frac{1}{2}\right)$ or $(a,b) \subset \left(\frac{1}{2},1\right)$ is an open problem.

\section{Intrinsic ergodicity for centred holes and the lexicographic world}
\label{intrinsicergodicity}

\noindent During the rest of the paper, we will consider centred holes $(a,b)$ such that $\dim_{H}(X_{(a,b)}) > 0$ only -see \cite[Lemma 1.1, Theorem 2.13]{sidorov1}-. Recall that $$D_1 = \left\{(a,b) \in \left(\frac{1}{4}, \frac{1}{2}\right) \times \left(\frac{1}{2}, \frac{3}{4}\right): \dim_{H}(X_{(a,b)}) > 0 \right\}.$$ We ask that $(a,b) \in D_1$ since $$\dim_{H}(X_{(a,b)}) = \dfrac{h_{top}(f_{(a,b)})}{\lambda}$$ where $\lambda$ is the Lyapunov exponent of $2x \mod 1$ and since systems with zero topological entropy are not intrinsically ergodic unless they are uniquely ergodic -see \cite{walters}. Note that $\dim_{H}(X_{(a,b)}) = \dim_{H}(\Lambda_{(a,b)})$. Also, from \cite[Corollary 3.21]{yomero3} we know that for every $(a,b) \in D_1$ there exist $(\alpha,\beta) \in \mathcal{LW}$ such that $(\Lambda_{(a,b)}, f_{(a,b)})$ is topologically conjugated to $(\Sigma_{(\alpha,\beta)}, \sigma_{(\alpha,\beta)})$. Also, in Hare and Sidorov \cite{sidorov6} defined the set $$D_2 = \left\{(a,b) \in D_1 : Bad(
a,b) \hbox{\rm{ is finite }} \right\}$$ where $$Bad(a,b) = \left\{n \geq 3 : f^n(x) \in (a,b) \hbox{\rm{ for }} x \in Per(f) \right\}.$$ It was shown in \cite{sidorov6} that $D_2 \subsetneq D_1$. 

\subsection*{Open dynamical systems and Lorenz maps}

Recall that a map $g:[0,1] \to [0,1]$ is said to be a \textit{Lorenz map} if there exists $c \in (0,1)$ such that:
\begin{enumerate}[i)]
\item $g\mid_{[0,c)}$ and $g\mid_{(c,1]}$ are continuous and strictly increasing;
\item $\mathop{\lim}\limits_{x^+ \to c}g(x) = 1$ and $\mathop{\lim}\limits_{x^- \to c}g(x) = 0$;
\item $\overline{I_c} = [0,1]$ where $I_c = \mathop{\bigcup}\limits_{n=0}^{\infty}g^{-n}(c)$.
\end{enumerate}

\vspace{1em} As such, the properties of expanding Lorenz maps been extensively studied - see e.g. \cite{glendinning, glendinning1, hubbardsparrow} among others. Also in \cite[6.2]{glendinning1} it was stated that Lorenz maps can be studied as open dynamical systems of the doubling map. Moreover, we associate to such dynamical systems a symbolic space via \textit{kneading theory} - see \cite{hall} for a suitable introduction to our context - which coincides with the lexicographic world. Parry in \cite{parry3} introduced the following linear and expanding Lorenz maps known as \textit{$\mod 1$ transformations}. Consider $\tilde{\beta} \in (1,2)$ and $\tilde{\alpha} \in (0,2-\beta)$. Define $$g_{\tilde{\beta},\tilde{\alpha}}(x) = \left\{
\begin{array}{clrr}      
\tilde{\beta} x + \tilde{\alpha}, & \hbox{\rm{ if }}  x \in [0, \frac{1 -\tilde{\alpha}}{\tilde{\beta}}];\\
\tilde{\beta} x + \tilde{\alpha} - 1 & \hbox{\rm{ if }} x \in [\frac{1 -\tilde{\alpha}}{\tilde{\beta}}, 1].\\
\end{array}
\right.$$ 
Thus, for every $\mod 1$ transformation $g_{(\alpha, \beta)}$, there is a centred hole $(a,b)$ such that $\left([0,1],g_{(\tilde{\beta}, \tilde{\alpha})}\right)$ is a factor of $\left(\Lambda_{(a,b)}, f_{(a,b)}\right)$ \cite[Proposition 4.1, Proposition 4.2.]{yomero3}. Also, the factor map is given by $\pi_{\tilde{\beta},\tilde{\alpha}}\circ \pi^{-1}$ where $\pi_{\tilde{\beta},\tilde{\alpha}}: \Sigma_{(\alpha,\beta)} \to [0,1]$ given by $\pi_{\tilde{\beta}, \tilde{\alpha}}((x_i)_{i=1}^{\infty}) = \frac{\tilde{\alpha}}{\tilde{\beta}-1} + \mathop{\sum}\limits_{n=1}^{\infty}\frac{x_n}{\tilde{\beta}^n}$.
 
\vspace{1em}Urbanski in \cite[Lemma 1]{urbanski2} showed that the projection map $\pi_{\tilde{\beta},\tilde{\alpha}}$ is also a measure theoretic isomorphism between $\mathbb{M}(\sigma_{(\alpha,\beta)})$ and $\mathbb{M}(g_{\beta, \alpha})$. On the other hand, Hobfbauer in \cite{hofbauer2, hofbauer1} it is shown that $(g_{\tilde{\beta}, \tilde{\alpha}})$ is intrinsically ergodic if $([0,1], g_{\tilde{\beta}, \tilde{\alpha}})$ is topologically transitive. Topologically transitive attractors $(\Lambda_{(a,b)}, f_{(a,b)})$ as well as dynamical systems of the form $([0,1], g_{\beta, \alpha})$ are characterised using \textit{renormalisation}. A pair $(\alpha,\beta) \in \mathcal{LW}$ is said to be \textit{renormalisable} if there exist two words $\omega$ and $\nu$ and sequences $\{n^{\omega}_i\}_{i=1}^{\infty}$ $\{n^{\nu}_i\}_{i=1}^{\infty}$, $\{m^{\omega}_j\}_{j=1}^{\infty}$ and $\{m^{\nu}_j\}_{j=1}^{\infty} \subset \mathbb{N}\cup \{\infty\}$ such that $\omega = 0-\max_{\omega}$, $\nu = 1-\min_{\nu}$, $(0-\max_{\nu})^{\infty} \prec \omega^{\infty}$, $\nu^{\infty}\prec(1-\min_{\omega})^{\infty}$ ,$\ell(\omega\nu) \geq 3$ and $$0\alpha = \omega \nu^{n^{\nu}_1}\omega^{n^{\omega}_1}\nu^{n^{\nu}_2}\omega^{n^{\omega}_2}\nu^{n^{\nu}_3}\ldots$$ and $$1\beta = \nu \omega^{m^{\omega}_1}\nu^{m^{\nu}_1}\omega^{m^{\omega}_2}\nu^{m^{\nu}_2}\omega^{n^{\omega}_3}\ldots .$$ If $\ell(\omega \nu) = 3$ we say that $(\alpha,\beta)$ is trivially renormalisable. The pair $(\omega, \nu)$ is called \textit{the associated pair of $(\alpha, \beta)$}. We will always consider the shortest choice of $(\omega, \nu)$ with respect to $\ell(\omega)$ and $\ell(\nu)$. It is well known that a Lorenz map $([0,1], g_{\beta, \alpha})$ is transitive if and only if $(\alpha, \beta)$ is not renormalisable \cite[Theorem 2]{glendinning}. Recently, the author shown in \cite[Theorem 5.16,Theorem 5.22]{yomero3} that for $(a,b) \in D_2$, $(\Lambda_{(a,b)}, f_{(a,b)})$ is transitive if and only if $(\alpha, \beta)$ is not renormalisable and that the shift given by 
$\alpha = \omega\nu^{\infty}$ and $\beta = \nu\omega^{\infty}$ is transitive if $\omega = \xi_r$ and $\nu = \zeta_r$ for $r \in \mathbb{Q}\cap(0,1)$ \cite[Theorem 5.10]{yomero3}. 

\subsection*{Intrinsic Ergodicity in $D_2$}We now start describing the pairs $(a,b) \in D_1$ such that $(\Lambda_{(a,b)},f_{(a,b)})$ is intrinsically ergodic. We would like to recall our main theorem.

\begin{theorem}
The set $$D_{I} = \left\{(a,b) \in D_1 : (\Lambda_{(a,b)},f_{(a,b)}) \hbox{\rm{ is intrinsically ergodic }} \right\}$$ has full Lebesgue measure. \label{elmerochingon}
\end{theorem}

As a consequence of \cite[Theorem 6.2, Theorem 6.3, Theorem 6.4, Theorem 6.5]{yomero1} and \cite[Theorem 5.20, Theorem 5.22, Theorem 6.2]{yomero3} combined with the fact that transitive subshifts of finite type, as well as subshifts with the specification property are intrinsically ergodic - see \cite{parry, bowen1} respectively- and \cite[Theorem 8]{hofbauer1} we obtain that $D_I \neq \emptyset$.

\vspace{1em}We say that a pair $(\alpha, \beta) \in \mathcal{LW}$ is \textit{essential pair} if $(\alpha, \beta)$ is not renormalisable and $\alpha$ and $\beta$ are periodic sequences. From \cite[Theorem 1.3]{samuel}, the lexicographic subshifts corresponding to essential pairs are transitive subshifts of finite type. Let $$\mathcal{E} = \left\{(\alpha, \beta) \in \mathcal{LW} : (\alpha, \beta) \hbox{\rm{ is an essential pair}}\right\}$$ and $$\mathcal{S} = \left\{(a,b) \in D_1 : (\Lambda_{(a,b)}), f_{(a,b)}) \hbox{\rm{ is conjugated to a transitive subshift of finite type}} \right\}.$$ Observe that $\pi(\mathcal{E}) \subset \mathcal{S} \subset D_1$ and $\pi(\mathcal{E}) \subset D_2$ \cite[Theorem 3.8]{sidorov6}, \cite[Theorem 2.13]{sidorov1}.  

\begin{lemma}
The set $\mathcal{S}$ has positive Lebesgue measure. \label{medidapositiva}
\end{lemma}
\begin{proof}
Let $(\alpha, \beta) \in \mathcal{LW}$ be an essential pair and $(a,b) = (\pi(0\alpha),\pi(1\beta))$. From \cite[Theorem 3.3]{yomero3} there is an open set $U \subset D_1$ such that for every $(a^{\prime}, b^{\prime}) \in U$, $X_{(a,b)} = X_{(a^{\prime}, b^{\prime})}$. This implies that $(\Lambda_{(a, b)}, f_{(a, b)})$ and $(\Lambda_{(a^{\prime}, b^{\prime})}, f_{(a^{\prime}, b^{\prime})})$ are topologically conjugated. Since $(\Lambda_{(a, b)}, f_{(a, b)})$ is conjugated to $(\Sigma_{(\alpha, \beta)}, \sigma_{(\alpha, \beta)})$ then $(\Lambda_{(a^{\prime}, b^{\prime})}, f_{(a^{\prime}, b^{\prime})})$ is also topologically conjugated to $(\Sigma_{(\alpha, \beta)}, \sigma_{(\alpha, \beta)})$ for every $(a^{\prime}, b^{\prime}) \in U$. This shows our result.
\end{proof}

Note that Lemma \ref{medidapositiva} implies that $D_I$ has positive Lebesgue measure. 

\vspace{1em}In \cite[Corollary 3.5]{yomero3} is shown that set $$T = \left\{(a,b) \in R : h_{top}(f_{(a,b)}) \hbox{\rm{ is constant}}\right\}$$ has full measure. Note that the set $$\mathcal{S}^{\prime} = \left\{(a,b) \in D_1 : (\Lambda_{(a,b)}), f_{(a,b)}) \hbox{\rm{ is conjugated to a subshift of finite type}} \right\} \subset T$$ and by Theorem \cite[Corollary 3.4]{yomero3} has full measure. Thus, to prove our result it suffices to show that every lexicographic subshift of finite type $(\Sigma_{(\alpha, \beta)}, \sigma_{(\alpha, \beta)})$ is intrinsically ergodic.

\subsubsection*{Renormalisation boxes and intrinsic ergodicity} 

\begin{definition}
Let $(\omega,\nu)$ be the associated pair of an essential pair $(\alpha^{\prime}, \beta^{\prime})$. The set $$\mathcal{B}^0(\omega,\nu) = \left\{(\alpha, \beta) \in \mathcal{LW} : (0\alpha, 1\beta) \in \left[\omega^{\infty}, \omega\nu^{\infty} \right]_{\prec} \times \left[\nu\omega^{\infty}, \nu^{\infty}\right]_{\prec}\right\}$$ is called a \textit{0-renormalisation box} or simply a \textit{renormalisation box}. \label{0renormalisationbox}
\end{definition}

\begin{lemma}
Let $\mathcal{B}^0(\omega, \nu)$ and $\mathcal{B}^0(\omega^{\prime}, \nu^{\prime})$ be renormalisation boxes. Then, $$\mathcal{B}^0(\omega, \nu) \cap \mathcal{B}^0(\omega^{\prime}, \nu^{\prime}) = \emptyset$$ if and only if $(\omega, \nu) \neq (\omega^{\prime}, \nu^{\prime})$. \label{cajas1}
\end{lemma}
\begin{proof}
It is clear that $\mathcal{B}^0(\omega, \nu) \cap \mathcal{B}^0(\omega^{\prime}, \nu^{\prime}) = \emptyset$ then $(\omega, \nu) \neq (\omega^{\prime}, \nu^{\prime})$. Let us assume now that $(\omega, \nu) \neq (\omega^{\prime}, \nu^{\prime})$. Then, there is no loss in generality assuming that ${\omega^{\prime}}^{\infty} \preccurlyeq {\omega}^{\infty}$. 

\vspace{1em}Let us assume that ${\omega^{\prime}}^{\infty} \prec \omega^{\infty}$ first. Then, there is $i \in \mathbb{N}$ such that ${\omega^{\prime}}^{\infty}_j = \omega^{\infty}_j$ for every $j \leq i$ and ${\omega^{\prime}}^{\infty}_i = 0$ and $\omega^{\infty}_i = 1$. Then we need to consider the following three cases. Suppose that $\ell(\omega) = \ell(\omega^{\prime})$. This implies that $i \leq \ell(\omega)$ then $\omega^{\prime}{\nu^{\prime}}^{\infty} \prec \omega^{\infty}$ regardless the choice of $\nu$ and $\nu^{\prime}$. Then, $\left[\omega^{\infty}, \omega\nu^{\infty}\right]_{\prec} \cap \left[{\omega^{\prime}}^{\infty}, \omega^{\prime}{\nu^{\prime}}^{\infty}\right]_{\prec} = \emptyset$ which implies our result. Now, let us assume that $\ell(\omega^{\prime}) < \ell(\omega)$. This gives $i \leq \ell(\omega)$, which implies that $$\left[\omega^{\infty}, \omega\nu^{\infty}\right]_{\prec} \cap \left[{\omega^{\prime}}^{\infty},\omega^{\prime}{\nu^{\prime}}^{\infty}\right]_{\prec} = \emptyset$$ as in the pre\-vious ca\-se. Sup\-po\-se now that $\ell(\omega) < \ell(\omega^{\prime})$. Then, we have to consider two sub-cases. If $i \leq \ell(\omega)$ then it is clear that $$\left[\omega^{\infty}, \omega\nu^{\infty}\right]_{\prec} \cap \left[{\omega^{\prime}}^{\infty}, \omega^{\prime}{\nu^{\prime}}^{\infty}\right]_{\prec} = \emptyset.$$ If $\ell(\omega)< i \leq \ell(\omega^{\prime})$ then $\omega^{\prime}{\nu^{\prime}}^{\infty} \prec \omega^{\infty}$ since $\omega^{\infty}_i = 1$ and $\omega^{\prime}{\nu^{\prime}}^{\infty}_i = 0$. Then our result follows.

If $\omega = \omega^{\prime}$ then $\nu \neq \nu^{\prime}$. It is clear that $(\bar{\nu}, \bar{\omega})$ and $(\bar{\nu^{\prime}}, \bar{\omega^{\prime}})$ define renormalisation boxes as well. Without losing generality we can assume that $\nu^{\infty} \prec {\nu^{\prime}}^{\infty}$. Then $\overline{{\nu^{\prime}}^{\infty}} \prec \overline{\nu^{\infty}}$. Then, applying the arguments shown above it we get $$\left[{\overline{\nu}}^{\infty}, \overline{\nu}{\overline{\omega}}^{\infty} \right]_{\prec} \cap \left[{\overline{\nu^{\prime}}}^{\infty}, \overline{\nu^{\prime}}{\overline{\omega^{\prime}}}^{\infty}\right]_{\prec} = \emptyset.$$ Therefore $$\left[\nu{\omega}^{\infty}, {\nu}^{\infty}\right]_{\prec} \cap \left[\nu^{\prime}{\omega^{\prime}}^{\infty}, {\nu^{\prime}}^{\infty}\right]_{\prec} = \emptyset.$$ This concludes the proof.
\end{proof}

Observe that $\Sigma_2\times \Sigma_2$ is a metric space with the distance given by $d_2:\left(\Sigma_2\times \Sigma_2\right) \times \left(\Sigma_2 \times \Sigma_2\right) \to \left[0,1\right]$ given by $$d_2((\alpha, \beta), (\alpha^{\prime}, \beta^{\prime})) = \left({d(\alpha, \alpha^{\prime})}^2 + d(\beta, \beta^{\prime})^{2} \right)^{\frac{1}{2}}.$$ Note that $$\hbox{\rm{diam}}(\mathcal{B}^0(\omega, \nu)) = d_2\left((\omega^{\infty}, \nu^{\infty}),(\omega\nu^{\infty}, \nu\omega^{\infty})\right)$$ and $$d_2\left((\omega^{\infty}, \nu^{\infty}),(\omega\nu^{\infty}, \nu\omega^{\infty})\right) = \left(\frac{1}{{2^{(\ell(\omega)+1)}}^2} + \frac{1}{{(2^{\ell(\nu)+1)}}^{2}} \right)^{\frac{1}{2}}.$$

\begin{theorem}
If $(\alpha,\beta) \in \mathcal{LW}$ satisfies that $(0\alpha,1\beta) \in \mathcal{B}^0(\omega, \nu) \setminus \left\{(\omega^{\infty}, \nu^{\infty})\right\}$, then $(\alpha,\beta)$ is a renormalisable pair by $\omega$ and $\nu$.\label{lemacajas1}
\end{theorem}
\begin{proof}
Let $(\alpha, \beta)$ satisfying our hypothesis. Since $(\alpha, \beta)$ is not an essential pair, from \cite[Theorem 5.21]{yomero3} then $(\alpha, \beta)$ is either coded but not of finite type or renormalisable. 

If $(\alpha,\beta)$ is renormalisable, then there are words $\omega^{\prime}$ and $\nu^{\prime}$ which renormalise $(\alpha, \beta)$. Assume that $(\omega, \nu) \neq (\omega^{\prime}, \nu^{\prime})$. Then by Lemma \ref{cajas1}, $(0\alpha, 1\beta) \notin \mathcal{B}^{0}(\omega, \nu)$ which is a contradiction. Thus $(\alpha,\beta)$ is renormalisable by $\omega$ and $\nu$. 

Assume that $(\alpha,\beta)$ is coded but not of finite type and let $\mathcal{B}^0(\omega, \nu)$ such that $(0\alpha,1\beta) \in \mathcal{B}^0(\omega, \nu)$. Then, there is a sequence $\left\{(\alpha_n, \beta_n)\right\}_{n=1}^{\infty}$ of essential pairs such that $$\Sigma_{(\alpha, \beta)} = \mathop{\lim}_{n\to \infty} \Sigma_{(\alpha_n, \beta_n)} \hbox{\rm{ with }} \alpha_n \prec \alpha_{n+1} \hbox{\rm{ and }} \beta_{n+1} \prec \beta_n$$ for every $n \in \mathbb{N}$. Since $(\Sigma_{(\alpha \beta)})$ is coded we have $\Sigma_{(\alpha_n, \beta_n)} \subset \Sigma_{(\alpha_{n+1}, \beta_{n+1})}$. Let $(\omega_n, \nu_n)$ be the associated pair of $(\alpha_n, \beta_n)$ for every $n$. We claim that either $\left\{\ell(\omega_n)\right\}_{n=1}^{\infty}$ is an unbounded sequence or $\left\{\ell(\nu_n)\right\}_{n=1}^{\infty}$ is an unbounded sequence. Assume that both sequences $\left\{\ell(\omega_n)\right\}_{n=1}^{\infty}$ and $\left\{\ell(\nu_n)\right\}_{n=1}^{\infty}$ are bounded. Let $M$ and $N \in \mathbb{N}$ such that $\left\{\ell(\omega_n)\right\} \leq M$ and $\left\{\ell(\nu_n)\right\} \leq M$. Observe that the set $$\left\{(\omega, \nu) \in \mathop{\bigcup}\limits_{n=2}^N B_n(\Sigma_2)\times \mathop{\bigcup}\limits_{n=2}^{M} B_m(\Sigma_2): (\alpha,\beta) \hbox{\rm{ is essential }}\right\}$$ is finite, then there is $n$ such that $(\alpha_n, \beta_n) = (\alpha,\beta)$. This gives that $(\alpha, \beta) = (\sigma(\omega^{\infty}), \sigma(\nu^{\infty}))$ which is a contradiction. 
\end{proof}

Observe that Theorem \ref{lemacajas1} gives us a decomposition of $\pi^{-1}(D_2)$, namely $(\alpha, \beta)$ is either essential, renormalisable or a coded system. Denote by $$\mathcal{C} = \left\{(\alpha, \beta) \in \mathcal{LW} : (\alpha, \beta) \hbox{\rm{ is coded}}\right\},$$ $$\mathcal{LW}^{\prime} = \left\{(0\alpha, 1\beta) : (\alpha,\beta) \in \mathcal{LW}\right\}$$ and $$\mathcal{C}^{\prime} = \left\{(0\alpha, 1\beta) : (\alpha,\beta) \in \mathcal{C} \right\}.$$ As a consequence of Theorem \ref{lemacajas1} and Lemma \ref{medidapositiva} we have that all the dynamical properties of $D_2$ are determined by essential pairs and limits of essential pairs. 

\begin{corollary}
$$\pi^{-1}(D_2) \cap \mathcal{LW}^{\prime} = \left(\mathop{\bigcup}\limits_{(\alpha,\beta) \in \mathcal{E}} \mathcal{B}^0(\omega, \nu)\right) \cup \mathcal{C}^{\prime}.$$ \label{lascajascubren1}
\end{corollary}

Moreover, as a consequence of \cite[Corollary 3.21]{yomero3} we have that every attractor $\left(\Lambda_{(a,b)}, f_{(a,b)}\right)$ with $(a,b) \in D_2$ is associated to an essential pair, a renormalisable subshift or to a coded system.

\subsubsection*{Transitive components of $(\Sigma_{(\alpha,\beta)},\sigma_{(\alpha,\beta)})$ for $(0\alpha,1\beta) \in D_2$}

Let $(\Sigma_{(\alpha,\beta)},\sigma_{(\alpha,\beta)})$ be a lexicographic subshift. A subset $A$ of $\Sigma_{(\alpha,\beta)}$ is said to be a \textit{transitive component} if $A$ is closed, completely invariant (i.e. $f^{-1}(A) = A = f(A)$), $\sigma\mid_A: A \to A$ is topologically transitive and there is no other set $A^{\prime}$ such that $A \subsetneq A^{\prime}$ containing a dense orbit \cite[p. 1313]{bundfuss}. Observe that \cite[Theorem 6.3]{bundfuss} and \cite[Corollary 3.21]{yomero3} imply that $\Sigma_{(\alpha,\beta)}$ has in general at most $8$ transitive components and that $\Sigma_{(\alpha,\beta)}$ has at most $4$ transitive components if both $\alpha$ and $\beta$ are periodic sequences. 

\vspace{1em}We characterise now such transitive components for $(\alpha, \beta)$ such that $(0\alpha, 1\beta) \in \mathcal{B}^0(\omega, \nu)$. Recall that, given two finite words $\omega$ and $\nu$, $\left\{\omega, \nu\right\}^{\infty}$ denotes the set of free concatenations of $\omega$ and $\nu$ and their shifts. 

\begin{lemma}
Let $(0\alpha, 1\beta) \in \mathcal{B}^0(\omega, \nu)$ be a renormalisable pair such $\left(\Sigma_{(\alpha,\beta)}, \sigma_{(\alpha,\beta)}\right)$ is a subshift of finite type. Then $\left(\Sigma_{(\alpha, \beta)},\sigma_{(\alpha, \beta)}\right)$ has two transitive components only. Moreover, the transitive components of $\left(\Sigma_{(\alpha,\beta)}, \sigma_{\alpha,\beta}\right)$, $A$ and $B$ are given by $A = \Sigma_{(\sigma(\omega^{\infty}), \sigma(\nu^{\infty}))}$ and $B \subsetneq \left\{\omega,\nu\right\}^{\infty}$. \label{prepara1}
\end{lemma}
\begin{proof}
From \cite[Theorem 5.13]{yomero3} gives that $\left(\Sigma_{(\alpha,\beta)}, \sigma_{(\alpha,\beta)}\right)$ is not transitive. Moreover, from the proof of \cite[Theorem 5.13]{yomero3} we obtain that there is no bridge between the words $\omega1$ and $\nu^{n_1^{\nu}+1}$. Also, it is possible to show that there is no bridge between $\nu0$ and $\omega^{m_1^{\omega}+1}$. Thus, $A = \Sigma_{(\sigma(\omega^{\infty}), \sigma(\nu^{\infty}))}$ is a transitive component of $\Sigma_{(\alpha, \beta)}$ since $(\sigma(\omega^{\infty}), \sigma(\nu^{\infty}))$ is essential. Observe that $$C = \left\{x \in \Sigma_{(\alpha, \beta)} : \omega1 \hbox{\rm{ or }} \nu0 \hbox{\rm{ occurs in }} x \right\}$$ determines a transitive component $B$. We construct such component $B$ as follows: firstly, note that $C$ is not necessarily an invariant set. Let $B$ the maximal invariant set of $C$, that is $$B = \mathop{\bigcap}\limits_{n = -\infty}^{\infty} \sigma^{n}(C).$$ Let us show that $B \subsetneq \left\{\omega, \nu\right\}^{\infty}$. 
Let $x \in B$. Without losing generality we can assume that $\omega1$ is a factor of $x$ and that $x_i = \omega_i$ for $i \in \left\{1, \ldots, \ell(\omega)\right\}$ and $x_{\ell(\omega)+1} = 1$. Observe that $x_{i+ \ell(\omega)} = \nu_{i}$ for $i \in \left\{1, \ldots, \ell(\nu)\right\}$ since $x \in \Sigma_{(\alpha, \beta)}$. Then $x_{\ell(\omega)+\ell(\nu)+1}$ is free. We claim that $x_{\ell(\omega)+\ell(\nu)+1} = 0$ then $x_{\ell(\omega)+ \ell(\nu)+i} = \omega_i$ for $i \in \left\{1, \ldots, \ell(\omega)\right\}$. Assume that the claim is false. Then there is $i \in \left\{1, \ldots, \ell(\omega)\right\}$ such that $\omega_i \neq x_{\ell(\omega)+ \ell(\nu)+i}$. If $\omega_i = 0$ and $x_{\ell(\omega)+ \ell(\nu)+i} = 1$ then $\sigma^{\ell(\omega)+ \ell(\nu)}(x) \succ \alpha$ which is a contradiction. Similarly, if $\omega_i = 1$ and $x_{\ell(\omega)+ \ell(\nu)+i} = 0$ then $\sigma^{\ell(\omega)}(x) \prec \beta$ which is a contradiction as well. We also claim that $x_{\ell(\omega)+\ell(\nu)+1} = 1$ then $x_{\
ell(\omega)+ \ell(\nu)+i} = \nu_i$ for $i \in \left\{1, \ldots, \ell(\omega)\right\}$. Assume that the claim is false. Then there is $i \in \left\{1, \ldots, \ell(\nu)\right\}$ such that $\omega_i \neq x_{\ell(\omega)+ \ell(\nu)+i}$. If $\nu_i = 0$ and $x_{\ell(\omega)+ \ell(\nu)+i} = 1$ then $\sigma^1(x) \succ \alpha$ which is a contradiction. Similarly, if $\nu_i = 1$ and $x_{\ell(\omega)+ \ell(\nu)+i} = 0$ then $\sigma^{\ell(\omega)}(x) \prec \beta$ which is a contradiction as well. Thus we have proven that $B \subset \left\{\omega,\nu\right\}^{\infty}$. To show that $B \subsetneq \left\{\omega,\nu\right\}^{\infty}$ observe that $\omega^{\infty} \notin B$. Also, it is possible to consider the sequence $\omega\nu^{\infty}$ and show that $\omega\nu^{\infty} \notin B$ as follows. Since $\left(\Sigma_{(\alpha, \beta)}, \sigma_{(\alpha, \beta)}\right)$ is a subshift of finite type, then $\alpha$ is periodic and $0\alpha = \omega\nu^{n_1^{\omega}}\omega^{n_1^{\omega}}\nu^{n_2^{\nu}}\ldots$. Thus $\sigma(\omega\
nu^{\infty}) \prec \alpha$. Therefore $\sigma(\omega\nu^{\infty}) \notin \Sigma_{(\alpha, \beta)}$. This shows that $B \subsetneq \left\{\omega, \nu\right\}^{\infty}$. Assume that there is another component $D$ such that $D \neq A$ and $D \neq B$. This implies that $D \cap (A \cup B) = \emptyset$. Thus, $D \subset C \cap B$. Since $C$ is the maximal invariant and $D$ is invariant we have that $D = \emptyset$ and the proof is finished.     
\end{proof}

\begin{lemma}
Let $(\alpha, \beta)$ satisfying the same hypothesis of Lemma \ref{prepara1}. Then $$h_{top}(\sigma_{B}) < h_{top}\left(\sigma_{\left\{\omega, \nu\right\}^{\infty}}\right).$$ \label{prepara2}
\end{lemma}
\begin{proof}
Since both $\alpha$ and $\beta$ are renormalisable periodic sequences then $$0\alpha = (\omega\nu^{n_1^{\nu}}\omega^{n_1^{\omega}}\nu^{n_2^{\nu}}\omega^{n_2^{\omega}} \ldots \nu^{n_i^{\nu}}\omega^{n_i^{\omega}})^{\infty}$$ and $$1\beta = (\nu\omega^{m_1^{\omega}}\nu^{m_1^{\nu}}\omega^{m_2^{\omega}}\nu^{m_2^{\nu}} \ldots \omega^{m_j^{\omega}}\nu^{m_j^{\nu}})^{\infty}.$$ Note that $n_i^{\omega}$ and $m_{j}^{\nu}$ can be equal to $0$. However, this will not affect our argument. Let $B$ the transitive component defined in Lemma \ref{prepara1}. Since $\left(\Sigma_{(\alpha, \beta)}, \sigma_{(\alpha, \beta)}\right)$ is a subshift of finite type $\left(B, \sigma_{(\alpha, \beta)_{B}}\right)$ is a subshift of finite type - see \cite[p. 1313]{bundfuss}. From \cite[Lemma 5.3]{yomero3} we have that the periodic orbit $(\omega\nu^{n_1^{\nu}+1})^{\infty} \in \Sigma_{(\alpha, \beta)} \setminus B$. Then, from \cite[Corollary 4.49]{lind} and Lemma \ref{prepara1} we have that $h_{top}(\sigma_B) < h_{top}\left(\sigma_{\left\{\omega, \nu\right\}^{\infty}}\right)$.
\end{proof}

\begin{theorem}
$\left(\Sigma_{(\alpha,\beta)}, \sigma_{(\alpha,\beta)}\right)$ has a unique transitive component of maximal entropy for every subshift of finite type $\left(\Sigma_{(\alpha,\beta)}, \sigma_{(\alpha,\beta)}\right)$ whenever $(0\alpha,1\beta) \in \pi^{-1}(D_2) \cap \mathcal{LW}^{\prime}$. \label{lemota}
\end{theorem}
\begin{proof}
Observe that if $(\alpha, \beta)$ is an essential pair then our result is automatically true. Suppose now that $(0\alpha, 1\beta) \in \mathcal{B}^0(\omega, \nu)$. Then the transitive components of $(\Sigma_{(\alpha,\beta)}, \sigma_{(\alpha,\beta)})$ are given by $\Sigma_{(\sigma(\omega^{\infty}), \sigma(\nu^{\infty}))}$ and the set $B$ constructed in Lemma \ref{prepara1}. From \cite[Proposition 2.5.5]{brin} and Lemma \ref{prepara1} we have that $h_{top}(\sigma_{(\alpha, \beta)}) = \max \left\{h_{top}(\sigma_{(\sigma(\omega^{\infty}, \nu^{\infty}))}), h_{top}(\sigma_B)\right\}$. From Lemma \ref{prepara2} we have that $h_{top}(\sigma_{B}) < h_{top}(\sigma_{\left\{\omega, \nu \right\}^{\infty}})$. Note that $h_{top}(\sigma_{\left\{\omega, \nu \right\}^{\infty}}) = \log\lambda$ where $\frac{1}{\lambda}$ is the unique root of $1 - t^{\ell(\omega)} - t^{\ell(\nu)}$ in $[0,1]$ - \cite[Remark 2.1]{hong} and \cite[p. 1008]{hall}. From \cite[Lemma 8, Lemma 9]{hall} we have that $h_{top}(\sigma_{(\sigma(\omega^{\infty}),\sigma(\nu^{\infty}))}) \geq h_{top}(\sigma_{\left\{\omega, \nu \right\}^{\infty}}) > h_{top}(\sigma_B)$ which proves our result.  
\end{proof}

\begin{corollary}
If $(\alpha, \beta) \in \mathcal{LW}$ such that $(0\alpha, 1\beta) \in \mathcal{B}^0(\omega, \nu)$ with $$(0\alpha, 1\beta) \neq (\omega\nu^{\infty}, \nu\omega^{\infty}),$$ then $(\Sigma_{(\alpha, \beta)}, \sigma_{(\alpha, \beta)})$ has a unique transitive component of maximal entropy.\label{notipofinito1}
\end{corollary}
\begin{proof}
Note that the results is automatically true if $(\alpha, \beta) = (\sigma(\omega^{\infty}),\sigma(\nu^{\infty}))$. From Theorem \ref{lemacajas1} we have that $(\alpha, \beta)$ is renormalisable by $\omega$ and $\nu$, i.e. $$0\alpha = \omega\nu^{n_1^{\nu}}\omega^{n_1^{\omega}}\nu^{n_2^{\nu}}\omega^{n_2^{\omega}} \ldots \nu^{n_i^{\nu}}\omega^{n_i^{\omega}}\ldots$$ and $$1\beta = \nu\omega^{m_1^{\omega}}\nu^{m_1^{\nu}}\omega^{m_2^{\omega}}\nu^{m_2^{\nu}} \ldots \omega^{m_j^{\omega}}\nu^{m_j^{\nu}}\ldots.$$  

Recall that $(\Sigma_{(\alpha, \beta)}, \sigma_{(\alpha, \beta)})$ is not a subshift of finite type. From \cite[Lemma 5.6]{yomero3} we have that  $(\omega\nu^{n_1^{\nu}})^{\infty} \prec 0\alpha \prec (\omega\nu^{n_1^{\nu}+1})^{\infty}$ and $(\nu\omega^{m_1^{\omega}+1})^{\infty} \prec 1\beta \prec (\nu\omega^{m_1^{\omega}})^{\infty}$. Let $B$ be the component defined in Lemma \ref{prepara1} for $(\Sigma_{(\alpha, \beta)}, \sigma_{(\alpha, \beta)})$ and $B^{\prime}$ and $B^{\prime \prime}$ the component constructed on Lemma \ref{prepara1} for $(\Sigma_{(\sigma((\omega\nu^{n_1^{\nu}})^{\infty}), \sigma((\nu\omega^{m_1^{\omega}})^{\infty}))}, \sigma_{(\sigma((\omega\nu^{n_1^{\nu}})^{\infty}), \sigma((\nu\omega^{m_1^{\omega}})^{\infty})})$ and $(\Sigma_{(\sigma((\omega\nu^{n_1^{\nu}+1})^{\infty}), \sigma((\nu\omega^{m_1^{\omega}+1})^{\infty})},\sigma_{(\sigma((\omega\nu^{n_1^{\nu}}+1)^{\infty}), \sigma((\nu\omega^{m_1^{\omega}+1})^{\infty})})$ respectively. Observe that $$h_{top}(\sigma_{B^{\prime}}) \leq h_{top}(
\sigma_{B}) \leq h_{top}(\sigma_{B^{\prime\prime}}).$$ Suppose that $(\Sigma_{(\alpha, \beta)}, \sigma_{(\alpha, \beta)})$ has a component $C$ such that $h_{top}(\sigma_C) = h_{top}(\sigma_{(\sigma(\omega^{\infty}),\sigma(\nu^{\infty}))})$. Observe that $C \neq B$ since $h_{top}(\sigma_C) > h_{top}(\sigma_B).$ Then neither $\omega1$ nor $\nu0$ are factors of $x$ for every $x \in C$. Then $$C = \left\{x \in \Sigma_{(\alpha, \beta)} : \hbox{\rm{ neither }}\omega1 \hbox{\rm{ nor }} \nu0 \hbox{\rm{ are factors of }} x\right\}.$$ Thus $C = \Sigma_{(\sigma(\omega^{\infty}), \sigma(\nu^{\infty}))}$.  
\end{proof}

\begin{theorem}
If $(\alpha, \beta)$ satisfies that $(0\alpha, 1\beta) \in \mathcal{B}^0(\omega, \nu)$ and $(\Sigma_{(\alpha, \beta)}, \sigma_{(\alpha, \beta)})$ is a subshift of finite type then $(\Sigma_{(\alpha, \beta)}, \sigma_{(\alpha, \beta)})$ is intrinsically ergodic. \label{ergodicidadintrinseca1}
\end{theorem}
\begin{proof}
From Theorem \ref{lemota} $\Sigma_{(\sigma(\omega^{\infty}), \sigma(\nu^{\infty}))}$ is the unique transitive component of maximal entropy of $(\Sigma_{(\alpha, \beta)}, \sigma_{(\alpha, \beta)})$. Since $(\sigma(\omega^{\infty}), \sigma(\nu^{\infty}))$ is an essential pair we have that $$(\Sigma_{(\sigma(\omega^{\infty}), \sigma(\nu^{\infty}))}, \sigma_{(\sigma(\omega^{\infty}), \sigma(\nu^{\infty}))}) \hbox{\rm{ is intrinsically ergodic.}}$$ 

Let $\mu$ the measure of maximal entropy for $(\Sigma_{(\sigma(\omega^{\infty}), \sigma(\nu^{\infty}))}, \sigma_{(\sigma(\omega^{\infty}), \sigma(\nu^{\infty}))})$. We define $$\mu_{(\alpha, \beta)}(B) = \mu(B \cap  \Sigma_{(\sigma(\omega^{\infty}), \sigma(\nu^{\infty}))})$$ for every measurable set $B$. Note that $h_{\mu_{(\alpha, \beta)}} = h_{\mu}$ since $h_{top}(\sigma_{(\sigma(\omega^{\infty}), \sigma(\nu^{\infty}))})=h_{top}(\sigma_{(\alpha, \beta)})$. Thus, $\mu_{(\alpha,\beta)}$ is a measure of maximal entropy with $\hbox{\rm{supp}}(\mu_{(\alpha, \beta)}) = \hbox{\rm{supp}}(\mu)$ where $\hbox{\rm{supp}}$ denotes the support of the measures. Suppose that there is another measure of maximal entropy $\eta$. Since $\eta \neq \mu_{(\alpha, \beta)}$ and $\eta, \mu_{(\alpha, \beta)}$ are ergodic measures \cite[Theorem 8.7]{walters} we have that $\eta$ and $\mu_{(\alpha, \beta)}$ are mutually singular \cite[Theorem 6.10]{walters}. This implies that $$\mu_{(\alpha,\beta)}(\hbox{\rm{supp}}(\mu_{(\alpha,\beta)}) \cap \hbox{\rm{supp}}(\eta)) = \eta(\hbox{\rm{supp}}(\mu_{(\alpha,\beta)}) \cap \hbox{\rm{supp}}(\eta)) = 0.$$ This gives that $\hbox{\rm{supp}}(\eta) \subset \Sigma_{(\alpha,\beta)} \setminus \Sigma_{(\sigma(\omega^{\infty}),\sigma(\nu^{\infty}))}.$ Then by Theorem \ref{lemota} $h_{\eta} = h_{top}(\sigma_B)$. Which contradicts that $\eta$ is a measure of maximal entropy and our proof is complete.    
\end{proof}

As a consequence of Corollary \ref{notipofinito1}, we can extend Theorem \ref{ergodicidadintrinseca1} to subshifts which are not of finite type necessarily. This is stated in the following corollary. The proof is left to the reader.  

\begin{corollary}
Let $(\alpha, \beta)$ be a renormalisable pair by $\omega$ and $\nu$ such that $(0\alpha, 1\beta) \in \mathcal{B}^0(\omega, \nu)$ with $(0\alpha, 1\beta) \neq (\omega\nu^{\infty}, \nu\omega^{\infty})$. Then $(\Sigma_{(\alpha, \beta)}, \sigma_{(\alpha, \beta)})$ is intrinsically ergodic.\label{ergodicidadintrinseca2}
\end{corollary}

As a corollary of Theorem \ref{ergodicidadintrinseca1} and Corollary \ref{ergodicidadintrinseca2} we obtain the following statement.

\begin{corollary}
For almost every $(a,b) \in D_2$, $(\Lambda_{(a,b)}, f_{(a,b)})$ is intrinsically ergodic. \label{d2}
\end{corollary}
\begin{proof}
Since $\mathcal{S}$ is open and dense in $D_1$ and $D_2$ is an open subset \cite{sidorov6} then $\mathcal{S} \cap D_2$ is open and dense in $D_2$. Then, by Lemma \ref{lascajascubren1}, Theorem \ref{ergodicidadintrinseca1} we can conclude our result.
\end{proof}

\subsubsection*{$\boldsymbol{D_I \neq D_2}$}

On Corollary \ref{notipofinito1} and Corollary \ref{ergodicidadintrinseca2} we made the assumption that $(0\alpha,1\beta) \neq (\omega\nu^{\infty}, \nu\omega^{\infty})$. We would like to explain the reason of this hypothesis. 

\vspace{1em}Let $\omega =  01$ and $\nu_k = 100(10)^k$ for some $k \geq 0$ and consider $(\alpha_k,\beta_k) \in \mathcal{LW}$ given by $$0\alpha = \omega\nu_k^{\infty} \hbox{\rm and }1\beta = \nu_k\omega^{\infty};$$ or $$0\alpha = \overline{\nu_k}{\overline{\omega}}^{\infty} \hbox{\rm and }1\beta = \overline{\omega}{\overline{\nu_k}}^{\infty}$$ where $\bar{\omega}$ denotes the mirror image of $\omega$, that is $\bar{\omega}_i = 1-w_i$ for each $i \in \{1, \ldots \ell(\omega)\}$. The case when $k = 0$ was introduced by Hofbauer in \cite{hofbauer1} and later on, it was extended for any $k \in \mathbb{N}$ by Glendinning and Hall in \cite{hall}. Observe that $(\omega, \nu_k)$ and $(\bar{\nu_k}, \bar{\omega})$ are associated pairs the essential pairs $$({\alpha_k}^{\prime},{\beta_k}^{\prime}) = \left(\sigma(\omega^{\infty}), \sigma(\nu_k^{\infty})\right)$$ and $$({\alpha_k}^{\prime \prime},{\beta_k}^{\prime \prime}) = \left(\sigma({\bar{\nu_k}}^{\infty}), \sigma({\bar{\omega}}^{\infty})\right).$$ As a consequence of \cite[Theorem 5.16]{yomero3} we have that $(\Sigma_{(\alpha_k,\beta_k)}, \sigma_{(\alpha_k,\beta_k)})$ is not a transitive subshift for any $k \geq 0$. To show that $(\Sigma_{(\alpha_k,\beta_k)}, \sigma_{(\alpha_k,\beta_k)})$ is not intrinsically ergodic for any $k \geq 0$, we need to use some results corresponding to Lorenz maps.

\vspace{1em}Recall that given a dynamical system $(X,f)$ we call a point $x \in X$ \textit{non wandering} if for every open set $U \subset X$ such that $x \in U$ there exists $n \in \mathbb{N}$ such that $f^n(U) \cap U \neq \emptyset$. We define \textit{the non-wandering set of $(X,f)$} to be $$\Omega(f) = \left\{x \in X : x \hbox{\rm{ is non wandering }} \right\}$$ \cite[p. 29]{brin}. In \cite[Corollary 10]{hall} is showed that if a Lorenz map $g$ has as kneading invariant $(\alpha,\beta) \in \mathcal{LW}$ satisfying that $0\alpha = \omega \nu_k^{\infty}$ and $1\beta = \nu_k\omega^{\infty}$ or $0\alpha = \overline{\nu_k}{\overline{\omega}}^{\infty}$ and $1\beta = \overline{\omega}{\overline{\nu_k}}^{\infty}$ then there are two basic components of the non-wandering set $A$ and $B$ such that $A\cap B = \emptyset$ and $h_{top}(g\mid_A) = h_{top}(g\mid_B) = h_{top}(g)$. 

\begin{theorem} 
Let $(\alpha,\beta) \in \mathcal{LW}$ satisfying that $$0\alpha = \omega \nu_k^{\infty} \hbox{\rm{ and }} 1\beta = \nu_k\omega^{\infty}$$ or $$0\alpha = \overline{\nu_k}{\overline{\omega}}^{\infty} \hbox{\rm and }1\beta = \overline{\omega}{\overline{\nu_k}}^{\infty}.$$ Then $(\Sigma_{(\alpha,\beta)}, \sigma_{(\alpha,\beta)})$ is not intrinsically ergodic.\label{canon1}
\end{theorem}
\begin{proof}
Let $(\alpha,\beta) \in \mathcal{LW}$ satisfying our hypothesis. From \cite[Theorem 1]{hubbardsparrow} there exists an expanding Lorenz map $g_{(\alpha,\beta)}$ such that $(\Sigma_{(\alpha,\beta)}, \sigma_{(\alpha,\beta)})$ is semi-conjugated to $([0,1], g_{(\boldsymbol{\alpha},\boldsymbol{\beta})})$ by a semi-conjugacy $h$. Let $\Omega_{(\boldsymbol{\alpha},\boldsymbol{\beta})}$ be the non-wandering set of $([0,1],g_{(\boldsymbol{\alpha},\boldsymbol{\beta})})$. Then by \cite[Corollary 10]{hall} there exist two $g_{(\boldsymbol{\alpha},\boldsymbol{\beta})}$-invariant sets $A,B \subset \Omega_{(\boldsymbol{\alpha},\boldsymbol{\beta})}$ such that $$h_{top}(g_{(\boldsymbol{\alpha},\boldsymbol{\beta})}\mid_A) = h_{top}(g_{(\boldsymbol{\alpha},\boldsymbol{\beta})}\mid_B) = h_{top}(g_{(\boldsymbol{\alpha},\boldsymbol{\beta})}).$$ By \cite[Theorem 2]{glendinning1}, $A \cap B = \emptyset$, whence $h^{-1}(A) \cap h^{-1}(B) = \emptyset$. Moreover, by \cite[Theorem 3]{hall}, $h_{top}(g_{(\boldsymbol{\alpha},\boldsymbol{\beta})}) = \log (\frac{1}{\kappa})$. Then by \cite[Lemma 3]{barnsley},  $h_{top}(\sigma_{(\boldsymbol{\alpha},\boldsymbol{\beta})}(h^{-1}(A))) = h_{top}(\sigma_{(\boldsymbol{\alpha},\boldsymbol{\beta})}(h^{-1}(B))) = h_{top}(\sigma_{(\boldsymbol{\alpha},\boldsymbol{\beta})})$. This implies that there exist two $\sigma_{(\boldsymbol{\alpha},\boldsymbol{\beta})}$-invariant measures $\mu_A$ and $\mu_B$ such that $\hbox{\rm{supp}}(\mu_A) \subset A$, $\hbox{\rm{supp}}(\mu_{B}) \subset B$ and $$h_{\mu_{A}} = h_{\mu_{B}} = h_{top}(\sigma_{(\boldsymbol{\alpha},\boldsymbol{\beta})}).$$ Thus, $(\Sigma_{(\boldsymbol{\alpha},\boldsymbol{\beta})}, \sigma_{(\boldsymbol{\alpha},\boldsymbol{\beta})})$ is not intrinsically ergodic. 
\end{proof}

Observe that from Theorem \ref{canon1} and \cite[Corollary 10]{hall} we obtain the following result.

\begin{corollary}
Let $(\alpha, \beta) = (\sigma(\omega\nu^{\infty}), \sigma(\nu\omega^{\infty}))$ where $\omega \neq 01$ and $\nu \neq 100(10)^k$ or $\omega \neq \overline{100(10)^k}$ or $\nu \neq 10$. Then $(\Sigma_{(\alpha, \beta)}, \sigma_{(\alpha, \beta)})$ is intrinsically ergodic.\label{ergodicidadintrinseca3}
\end{corollary} 

\subsection*{Intrinsic Ergodicity in $D_1 \setminus D_2$} 

To finish the proof of Theorem \ref{elmerochingon} it is necessary to prove that subshifts of finite type corresponding to pairs $(\alpha, \beta) \in \pi^{-1}(D_1 \setminus D_2) \cap \mathcal{LW}$ are intrinsically ergodic. As a consequence of \cite[Theorem 2.13]{sidorov1} and \cite[Theorem 3.8. Corollary 3.13]{sidorov6} we have that $(0\alpha, 1\beta) \in \left[\xi_r^{\infty}, \xi\zeta_r^{\infty} \right]_{\prec} \times \left[\zeta\xi_r^{\infty}, \zeta_r^{\infty} \right]_{\prec}$ for every $(\alpha, \beta) \in \pi^{-1}(D_1 \setminus D_2) \cap \mathcal{LW}$ with $r \in \mathbb{Q}\cap (0,1)$. Moreover, using a similar argument as in Theorem \ref{lemacajas1}, if $(\alpha, \beta) \in \pi^{-1}(D_1 \setminus D_2) \cap \mathcal{LW}$ then $(\alpha, \beta)$ is renormalisable by $(\xi_r, \zeta_r)$.

\begin{definition}
 Let $(\alpha,\beta) \in \mathcal{LW}$ be an essential pair with associated pair $(\omega,\nu)$ and consider the renormalisation box $\mathcal{B}^0(\omega,\nu)$ associated to $(\omega, \nu)$. Let $r \in \mathbb{Q} \cap (0,1)$. We define a \textit{(1,r)-renormalisation box} or simply a \textit{$1$-renormalisation box} to be $$\mathcal{B}^1_r(\omega, \nu) = \left\{(\alpha, \beta) \in \mathcal{LW} : (0\alpha, 1\beta) \in \rho_r(\mathcal{B}^0(\omega, \nu))\right\}.$$ Given a fi\-ni\-te se\-quen\-ce $(r_1, \ldots r_n) \in \left(\mathbb{Q}\cap(0,1)\right)^{n}$ we define an \textit{$(n-(r_1\ldots r_n))$-re\-nor\-ma\-li\-sa\-tion box} or simply an \textit{$n$- re\-nor\-ma\-li\-sa\-tion box} to be $$\mathcal{B}^n_{(r_1,\ldots r_n)}(\omega, \nu) = \left\{(\alpha, \beta) \in \mathcal{LW} : (0\alpha, 1\beta) \in \rho_{r_n} \circ \ldots \circ \rho_{r_1}(\mathcal{B}^0_{\omega, \nu})\right\}.$$ From Theorem \ref{cajas1} it is clear that for every $n \in \mathbb{N}$, $$B^n_{r_1 \ldots r_n}(\omega, \nu) \cap B^n_{r_1\ldots r_n}(\omega^{\prime}, \nu^{\prime}) = \emptyset$$ if $(\omega, \nu) \neq (\omega^{\prime}, \nu^{\prime})$. Similarly, it is clear that $$B^n_{r_1 \ldots r_n}(\omega, \nu) \cap B^n_{r_1^{\prime} \ldots r_n^{\prime}}(\omega, \nu) = \emptyset$$ if $(r_1, \ldots, r_n) \neq (r_1^{\prime}, \ldots r_n^{\prime})$ and that $$B^n_{r_1 \ldots r_n}(\omega,\nu) \cap B^m_{r_1 \ldots r_m}(\omega,\nu) = \emptyset$$ if $n \neq m$. \label{nrenormalisationbox}
\end{definition}

\vspace{1em}Note that $$\hbox{\rm{diam}}\left(\mathcal{B}_{r_1, \ldots r_n}^n(\omega, \nu)\right) = \left(\frac{1}{{2^{(q_1\cdot\ldots\cdot q_n \cdot\ell(\omega)+1)}}^2} + \frac{1}{{2^{(q_1\cdot\ldots\cdot q_n\cdot\ell(\nu)+1)}}^{2}} \right)^{\frac{1}{2}}.$$

\vspace{1em}We characterise now such transitive components for $(\alpha, \beta) \in \mathcal{LW}$ such that $(0\alpha, 1\beta) \in \mathcal{B}^n_{r_1, \ldots, r_n}(\omega, \nu)$. Note that \cite[Proposition 2.6]{sidorov1} implies that $\left\{\xi_{r_n}^{\infty}, \zeta_{r_n}^{\infty} \right\} \subsetneq \Sigma_{(\alpha, \beta)}$ for every $(\alpha, \beta)$ such that $(0\alpha, 1\beta) \in \mathcal{B}^n_{r_1, \ldots, r_n}(\omega, \nu)$.

\begin{theorem}
Let $(\alpha, \beta)$ be an essential pair with associated pair $(\omega, \nu)$. Let $(r_1, \ldots r_n) \in \left(\mathbb{Q}\cap(0,1)\right)^{n}$. Then, the subshift $\left(\Sigma_{(\alpha^{\prime}, \beta^{\prime})},\sigma_{(\alpha^{\prime}, \beta^{\prime})}\right)$ where $\alpha^{\prime}= \sigma(\rho_{r_n}\circ \ldots \circ \rho_{r_1}(\omega))$ and $\beta^{\prime} = \sigma(\rho_{r_n}\circ \ldots \circ \rho_{r_1}(\nu))$ has two transitive components $A^{\prime}$ and $B^{\prime}$ where $$A^{\prime} = \mathop{\bigcap}\limits_{m = - \infty}^{\infty}\sigma^m(\rho_{r_n}\circ \ldots \circ \rho_{r_1}(\Sigma_{(\alpha,\beta)}))$$ and $B^{\prime} = \left\{\xi_{r_n}^{\infty}, \zeta_{r_n}^{\infty} \right\}$\label{prepara1renorm}
\end{theorem}
\begin{proof}
Firstly we will show that $A^{\prime}$ is a transitive component. Recall that $r_i = \frac{p_i}{q_i}$ with $1 \leq i \leq n$. Let $\upsilon, \kappa \in \mathcal{L}(A^{\prime})$. Note that, there exist $l \leq \ell(\upsilon)$, $m \leq q_n$ such that and a word $u \in B_m(A^{\prime})$ such that $ \sigma^l(\upsilon)u \in \mathcal{L}(A^{\prime})$, $$\sigma^l(\upsilon)u = \xi_{r_n}^{n_1}\zeta_{r_n}^{n_2}\ldots\xi_{r_n}^{n_{i-1}}\zeta_{r_n}^{n_i}$$ where $\left\{n_j\right\}_{j=1}^i \subset \mathbb{N}$ such that $$\upsilon_{(\alpha,\beta)}=(\rho_{r_n}\circ \ldots \circ \rho_{r_1})^{-1}(\sigma^n(\upsilon)u) \in \mathcal{L}(\Sigma_{\alpha,\beta}).$$ Similarly, there are words $u^{\prime}, v^{\prime} \in \mathcal{L}(A^{\prime})$ such that $u^{\prime}\kappa v^{\prime} \in \mathcal{L}(A^{\prime})$ and $$u^{\prime}\kappa v^{\prime} = \xi_{r_n}^{n^{\prime}_1}\zeta_{r_n}^{n^{\prime}_2}\ldots\xi_{r_n}^{n^{\prime}_{k-1}}\zeta_{r_n}^{n^{\prime}_k}$$ where $\left\{n^{\prime}_j\right\}_{j=1}^k \subset \mathbb{N}$ such that $$\kappa_{(\alpha,\beta)}= (\rho_{r_n}\circ \ldots \circ \rho_{r_1})^{-1}(u^{\prime}\kappa v^{\prime}) \in \mathcal{L}(\Sigma_{\alpha,\beta}).$$  Since $(\alpha,\beta)$ is an essential pair, there exists a word $\eta$ such that $$\upsilon_{(\alpha, \beta)}\eta\kappa_{(\alpha, \beta)} \in \mathcal{L}(\Sigma_{(\alpha,\beta)}).$$ Then $$\rho_{r_n}\circ \ldots \circ \rho_{r_1}(\upsilon_{(\alpha, \beta)}\eta\kappa_{(\alpha, \beta)}) \in \mathcal{L}(A^{\prime}).$$ Since $$\rho_{r_n}\circ \ldots \circ \rho_{r_1}(\upsilon_{(\alpha, \beta)}\eta\kappa_{(\alpha, \beta)}) = \rho_{r_n}\circ \ldots \circ \rho_{r_1}(\upsilon_{(\alpha, \beta)})\rho_{r_n}\circ \ldots \circ \rho_{r_1}(\eta)\rho_{r_n}\circ \ldots \circ \rho_{r_1}(\kappa_{(\alpha, \beta)})$$ we have that $u\rho_{r_n}\circ \ldots \circ \rho_{r_1}(\eta)u^{\prime}$ is a bridge between $\upsilon$ and $\kappa$ with $$u(\rho_{r_n}\circ \ldots \circ \rho_{r_1}(\eta))u^{\prime} \in \mathcal{L}(A^{\prime}).$$ This shows that $A^{\prime}$ is a transitive component. 

To show that $B^{\prime}$ is a transitive component note that recall that $B^{\prime} \subset \Sigma_{(\alpha^{\prime}, \beta^{\prime})}$ for every $(\alpha^{\prime}, \beta^{\prime}) \in \mathcal{LW}$ such that $$(0\alpha^{\prime}, 1\beta^{\prime}) \in \mathcal{B}^n_{r_1 \ldots r_n}(\omega, \nu) .$$ Observe that $$0\alpha^{\prime} = \rho_{r_n}\circ \ldots \circ \rho_{r_1}(\omega) = \left(\xi_{r_n}\zeta_{r_n}^{n_1^{\zeta_{r_n}}} \ldots \xi_{r_n}^{n_i^{\xi_{r_n}}}\zeta_{r_n}^{n_i^{\zeta_{r_n}}}\right)^{\infty}$$ and $$1\beta^{\prime} = \rho_{r_n}\circ \ldots \circ \rho_{r_1}(\nu) =  \left(\zeta_{r_n}\xi_{r_n}^{m_1^{\xi_{r_n}}} \ldots \zeta_{r_n}^{m_j^{\zeta_{r_n}}}\xi_{r_n}^{m_j^{\xi_{r_n}}}\right)^{\infty}.$$ From  \cite[Lemma 5.6]{yomero3} we have that $$\max \left\{n_i^{\xi_{r_n}}\right\} \leq \max \left\{m_j^{\xi_{r_n}}\right\} = m_1^{\xi_{r_n}}$$ and $$\max \left\{m_j^{\zeta_{r_n}}\right\} \leq \max \left\{n_i^{\zeta_{r_n}} \right\} = n_1^{\zeta_{r_n}}.$$ Since $B^{\prime} \subset \Sigma_{(\alpha^{\prime}, \beta^{\prime})}$ and $B^{\prime}$ is $\sigma$-invariant we have that $\xi^{m_1^{\xi_{r_n}}+1}$ and $\zeta^{n_1^{\zeta_{r_n}}+1} \in \mathcal{L}(\Sigma_{(\alpha^{\prime}, \beta^{\prime})})$ and that $\xi^{m_1^{\xi_{r_n}}+1}$ and $\zeta^{n_1^{\zeta_{r_n}}+1}$ are factors of $x \in \Sigma_{(\alpha^{\prime}, \beta^{\prime})}$ if and only if $x \in B$. Thus $B^{\prime}$ is a transitive component. 

Assume that there is a third transitive component $C^{\prime}$ such that $C^{\prime} \neq A^{\prime}$ and $C^{\prime} \neq B^{\prime}$. Since $C^{\prime} \neq A^{\prime}$ we have that $x \notin (\rho_{r_n}\circ \ldots \circ \rho_{r_1})^{-1}(\Sigma_{(\alpha, \beta)})$. This implies that $\xi^{m_1^{\xi_{r_n}}+k}$ or $\zeta^{n_1^{\zeta_{r_n}}+k}$ are factors of $x$. On the other hand, since $x \notin B^{\prime}$ then neither $\xi^{m_1^{\xi_{r_n}}+k}$ nor $\zeta^{n_1^{\zeta_{r_n}}+k}$ are factors of $x$ for every $k \geq 1$. Thus $C^{\prime}  = \emptyset$. 
\end{proof}

\begin{corollary}
Let $(\alpha, \beta)$ be an essential pair with associated pair $(\omega, \nu)$. Let $(r_1, \ldots r_n) \in \left(\mathbb{Q}\cap(0,1)\right)^{n}$. Then, for every $(\alpha^{\prime}, \beta^{\prime})$ such that $(0\alpha^{\prime}, 1\beta^{\prime}) \in \mathcal{B}^n_{(r_1, \ldots, r_n)}(\omega, \nu)$such that the subshift $\left(\Sigma_{(\alpha^{\prime}, \beta^{\prime})},\sigma_{(\alpha^{\prime}, \beta^{\prime})}\right)$ is a subshift of finite type, the $\Sigma_{(\alpha^{\prime}, \beta^{\prime})}$ has three transitive components $A^{\prime}$, $B^{\prime}$ and $C^{\prime}$ where $$A^{\prime} = \mathop{\bigcap}\limits_{m=-\infty}^{\infty}\sigma^{m}\left(\rho_{r_n}\circ \ldots \circ \rho_{r_1}(A)\right),$$ $$B^{\prime} = \mathop{\bigcap}\limits_{m=-\infty}^{\infty}\sigma^{m}\left(\rho_{r_n}\circ \ldots \circ \rho_{r_1}(B)\right)$$ and $C^{\prime} = \left\{\xi_{r_n}^{\infty}, \zeta_{r_n}^{\infty} \right\}$ where $A$ and $B$ are the components constructed in Lemma \ref{prepara1}.\label{prepara2renorm}
\end{corollary}
\begin{proof}
Observe that since $(0\alpha^{\prime}, 1\beta^{\prime}) \in \mathcal{B}^n_{(r_1, \ldots, r_n)}(\omega, \nu)$ we have that $\left(\Sigma_{(\alpha,\beta)}, \sigma_{(\alpha,\beta)}\right)$  is a subshift of $\left(\Sigma_{(\alpha^{\prime}, \beta^{\prime})}, \sigma_{(\alpha^{\prime},\beta^{\prime})}\right)$. From Lemma \ref{prepara1} and using a similar argument as in Theorem \ref{prepara1renorm} we have that $A^{\prime}$, $B^{\prime}$ and $C^{\prime}$ are transitive components of $\left(\Sigma_{(\alpha, \beta)}, \sigma_{(\alpha,\beta)}\right)$. Assume that there is another transitive component $D^{\prime}$.  From Theorem \ref{prepara1renorm} we have that $A^{\prime} \neq D^{\prime}$ and $C^{\prime} \neq D^{\prime}.$ Let $x \in D$. This implies that $\xi^{m_1^{\xi_{r_n}}+k}$ or $\zeta^{n_1^{\zeta_{r_n}}+k}$ are factors of $x$. Thus $x \in B^{\prime}$. Then $D^{\prime}  \subset B^{\prime}$ which gives $D^{\prime} = B^{\prime}$ since $B^{\prime}$ and $D^{\prime}$ are transitive components. 
\end{proof}

\begin{proposition}
Let $(\alpha, \beta) \in \mathcal{LW}$ be an essential pair with associated pair $(\omega, \nu)$ and $r = \frac{p}{q} \in \mathbb{Q} \cap (0,1)$. Then $\rho_{r}: \Sigma_{(\alpha, \beta)} \to \Sigma_{(\alpha^{\prime}, \beta^{\prime})}$ is continuous and injective where $\alpha = \sigma(\rho_{r}(\omega^{\infty})) $ and $\beta = \sigma(\rho_{r}(\nu^{\infty})).$ \label{auxiliar1}
\end{proposition}
\begin{proof}
Let $\left\{x_n\right\}_{n=1}^{\infty} \subset \Sigma_{(\alpha, \beta)}$ such that $x_n \mathop{\longrightarrow}\limits_{n \to \infty} y$. Let $\varepsilon > 0$. Since $x_n \mathop{\longrightarrow}\limits_{n \to \infty} x$ there is $M \in \mathbb{N}$ such that for every $m \geq M$, $$d(x_m, y) < \frac{1}{2^M} < \varepsilon.$$ Thus, for every $m \geq M$, $(x_m)_i  = y_i$ for every $1 \leq i \leq M$. This implies that $\rho_{r}(x_m)_{i} = \rho_{r}(y)_i$ for every $1 \leq i \leq qM$. This gives that $$d(\rho_{r}(x_m), \rho_{r}(y)) \leq \frac{1}{2^{qM}} < \frac{1}{2^{M}} <\varepsilon.$$ To show that it is injective note that if $x = y$ then $\rho(x)_i = \rho(y)_i$ for every $1 \leq i \leq qk$ and $\rho(x)_{qk+1} = \rho(y)_{qk+1}$ where $k = \min{j \in \mathbb{N} : x_j \neq y_j}$.
\end{proof}

\begin{theorem}
Let $(\alpha, \beta)$ be an essential pair with associated pair $(\omega, \nu)$ and $(r_1, \ldots r_n) \in \left(\mathbb{Q}\cap(0,1)\right)^{n}$. Then $\left(\Sigma_{(\alpha^{\prime}, \beta^{\prime})},\sigma_{(\alpha^{\prime}, \beta^{\prime})}\right)$ has positive topological entropy where $$\alpha = \sigma(\rho_{r_n}\circ \ldots \rho_{r_1}(\omega^{\infty})) $$ and $$\beta = \sigma(\rho_{r_n}\circ \ldots \rho_{r_1}(\nu^{\infty})).$$ Moreover $$h_{top}(\sigma_{(\alpha^{\prime}, \beta^{\prime})}) = \frac{1}{q_1 \cdot \ldots \cdot q_n}h_{top}\left(\sigma_{(\alpha, \beta)}\right).$$ \label{lemadelaentropia}
\end{theorem}
\begin{proof}
From Theorem \ref{prepara1renorm} and \cite[Proposition 3.17 (2)]{katok}, we have that $$h_{top}(\sigma_{(\alpha^{\prime},\beta^{\prime})}) = \max\left\{h_{top}(\sigma\mid_{A^{\prime}}), h_{top}(\sigma\mid_{B^{\prime}})\right\}.$$ Moreover, from the construction of $A^{\prime}$ on Theorem \ref{prepara1renorm} and Proposition \ref{auxiliar1} we have that $$\rho_{r_n}\circ \ldots \circ \rho_{r_1}(\Sigma_{(\alpha, \beta)}) = A^{\prime}.$$ Let $x \in \Sigma_{(\alpha, \beta)}$ and $\varepsilon > 0$. Note that $$\rho_{r_n}\circ \ldots \circ \rho_{r_1}(B^{d}_{\varepsilon}(x)) = B^d_{(q_1\cdot \ldots \cdot q_n)\cdot \varepsilon}(\rho_{r_n}\circ \ldots \circ \rho_{r_1}(x)).$$ Thus, $\rho_{r_n}\circ \ldots \circ \rho_{r_1}$ is an open map and therefore an homeomorphism. Let $x = (x_i)_{i=1}^{\infty}\in \Sigma_{(\alpha, \beta)}$. Note that 
\begin{align*}
\rho_{r_n}\circ \ldots \circ \rho_{r_1}(\sigma(x)) &= \rho_{r_n}\circ \ldots \circ \rho_{r_1}((x_{i+1})_{i=1}^{\infty})\\
&= \sigma^{q_1\cdot \ldots \cdot q_n}(\rho_{r_n}\circ \ldots \circ \rho_{r_1}(x_{i})_{i=1}^{\infty}). 
\end{align*}
Thus, $\rho_{r_n}\circ \ldots \circ \rho_{r_1}$ is a conjugacy between $\left(\Sigma_{(\alpha, \beta)},\sigma_{(\alpha, \beta)}\right)$ and $\left(A^{\prime}, \sigma^{q_1\cdot \ldots \cdot q_n}\mid_{A^{\prime}}\right)$. Then from \cite[Proposition 3.17 (3)]{katok} we have that $$h_{top}(\sigma_{(\alpha, \beta)}) = h_{top}(\sigma^{q_1\cdot \ldots \cdot q_n}\mid_{A^{\prime}}) = q_1\cdot \ldots \cdot q_n\cdot h_{top}(\sigma\mid_{A^{\prime}}),$$ and $h_{top}(B^{\prime}) = 0$ the proof is complete. 
\end{proof}

\begin{lemma}
Let $(\alpha, \beta)$ such that $(0\alpha, 1\beta) \in B^n_{r_1, \ldots, r_n}(\omega,\nu)$ where $(\omega, \nu)$ is the associated pair of an essential pair $(\alpha^{\prime}, \beta^{\prime})$. Then $$h_{top}(\sigma_{B^{\prime}}) = \frac{1}{q_1\cdot \ldots \cdot q_n}h_{top}(B) \leq \frac{1}{q_1\cdot \ldots \cdot q_n}h_{top}\left(\sigma_{\left\{\omega,\nu\right\}^{\infty}}\right).$$ \label{prepara3renorm}
\end{lemma}
\begin{proof}
From Lemma \ref{prepara2} we have that $h_{top}(\sigma_B) \leq h_{top}(\sigma_{\left\{\omega,\nu\right\}^{\infty}}).$ From Corollary \ref{prepara2renorm} we have that $$B^{\prime} = \mathop{\bigcap}\limits_{m=-\infty}^{\infty}\sigma^{m}\left(\rho_{r_n}\circ \ldots \circ \rho_{r_1}(B)\right).$$ As in Theorem \ref{lemadelaentropia} $\rho_{r_n}\circ \ldots \circ \rho_{r_1}$ is a conjugacy between $(B, \sigma^{q_1\cdot \ldots \cdot q_n})$ and $(B^{\prime}, \sigma_{B^{\prime}})$. Then $$h_{top}(\sigma_{B^{\prime}}) = \frac{1}{q_1\cdot \ldots \cdot q_n}h_{top}(B) \leq \frac{1}{q_1\cdot \ldots \cdot q_n}h_{top}\left(\sigma_{\left\{\omega,\nu\right\}^{\infty}}\right)$$ holds. 
\end{proof}

\begin{theorem}
$\left(\Sigma_{(\alpha,\beta)}, \sigma_{(\alpha,\beta)}\right)$ has a unique transitive component of maximal entropy for every $\left(\Sigma_{(\alpha,\beta)}, \sigma_{(\alpha,\beta)}\right)$ subshift of finite type with $(\alpha, \beta) \in \mathcal{LW} \cap \pi^{-1}(D_1 \setminus D_2)$. \label{lemotarenorm}
\end{theorem}
\begin{proof}
From Theorem \ref{lemota} we have that $h_{top}(\sigma_{B})< h_{top}(\sigma_{A})$. Then by Corollary \ref{prepara2renorm}, Theorem \ref{lemadelaentropia} and Lemma \ref{prepara3renorm} we have that $A^{\prime}$ is the unique component of maximal entropy of $\left(\Sigma_{(\alpha,\beta)}, \sigma_{\alpha, \beta}\right)$.
\end{proof}

As a consequence of Corollary \ref{notipofinito1} and Theorem \ref{lemotarenorm} we obtain immediately the following result.

\begin{corollary}
Let $(\alpha, \beta)$ be a renormalisable pair by $\omega$ and $\nu$ such that $(0\alpha, 1\beta) \in \mathcal{B}^0(\omega, \nu)$ with $(0\alpha, 1\beta) \neq (\omega\nu^{\infty}, \nu\omega^{\infty})$. Then $(\Sigma_{(\alpha, \beta)}, \sigma_{(\alpha, \beta)})$ has a unique transitive component of maximal entropy.\label{notipofinito1renorm}
\end{corollary}

Observe that Theorem \ref{lemotarenorm} will imply that every lexicographic subshift of finite type is intrinsically ergodic as we state as follows.

\begin{theorem}
If $(\alpha, \beta)$ satisfies that $(0\alpha, 1\beta) \in \mathcal{B}^0(\omega, \nu)$ and $(\Sigma_{(\alpha, \beta)}, \sigma_{(\alpha, \beta)})$ is a subshift of finite type then $(\Sigma_{(\alpha, \beta)}, \sigma_{(\alpha, \beta)})$ is intrinsically ergodic. \label{ergodicidadintrinseca1renorm}
\end{theorem}

The proof of Theorem \ref{ergodicidadintrinseca1renorm} is an easy modification of the proof of Theorem \ref{ergodicidadintrinseca1}, so we will omit the proof. Moreover, as it was done in Corollary \ref{ergodicidadintrinseca2}, we can extend Theorem \ref{ergodicidadintrinseca1renorm} to subshifts which are not of finite type necessarily as it is stated in the following corollary. The proof is also omitted.  

\begin{corollary}
Let $(\alpha, \beta)$ be a renormalisable pair by $\omega$ and $\nu$ such that $(0\alpha, 1\beta) \in \mathcal{B}^0(\omega, \nu)$ with $(0\alpha, 1\beta) \neq (\omega\nu^{\infty}, \nu\omega^{\infty})$. Then $(\Sigma_{(\alpha, \beta)}, \sigma_{(\alpha, \beta)})$ is intrinsically ergodic.\label{ergodicidadintrinseca2renorm}
\end{corollary}
 
Recall that $$\mathcal{C} = \left\{(\alpha, \beta) \in \mathcal{LW}: \left(\Sigma_{(\alpha,\beta)}, \sigma_{(\alpha,\beta)}\right) \hbox{\rm{ is coded }}  \right\}$$ and $\mathcal{C}^{\prime} = \left\{(0\alpha, 1\beta) : (\alpha, \beta) \in \mathcal{C}\right\}$. 

\begin{theorem}
Let $n \in \mathbb{N}$ and $(r_1,\ldots,r_n)\in (\mathbb{Q}\cap(0,1))^{n}$. If $(0\alpha^{\prime}, 1\beta^{\prime}) \in \rho_{r_n}\circ \ldots \circ \rho_{r_1}(\mathcal{C}^{\prime})$ then $\left(\Sigma_{(\alpha^{\prime}, \beta^\prime)}, \sigma_{(\alpha^{\prime}, \beta^\prime)}\right)$ is intrinsically ergodic. \label{renormandcoded}
\end{theorem}
\begin{proof}
Let $(\alpha, \beta) \in \mathcal{C}$. Consider $A^{\prime} = \mathop{\bigcap}\limits_{m= -\infty}^{\infty}\sigma^{m}(\rho_{r_n}\circ\ldots\circ\rho_{r_1}(\Sigma_{(\alpha, \beta)}))$. Since $(\Sigma_{(\alpha,\beta)}, \sigma_{(\alpha,\beta)})$ is coded, a similar argument as the one used in Theorem \ref{prepara1renorm} implies that $A^{\prime}$ is a transitive component. Moreover, observe that $(\Sigma_{(\alpha, \beta)}, \sigma_{(\alpha,\beta)})$ is topologically conjugated to $(A^{\prime}, \sigma^{q_1\cdot \ldots \cdot q_n}_{A^{\prime}})$. Then, there is a unique measure of maximal entropy for $(A^{\prime}, \sigma^{q_1\cdot \ldots \cdot q_n}_{A^{\prime}})$. We call such measure $\mu_{A^{\prime}}$. As in Theorem \ref{ergodicidadintrinseca1} we extend such measure to a measure $\mu_{(\alpha^{\prime}, \beta^{\prime})} \in \mathbb{M}(\sigma_{(\alpha^{\prime}, \beta^{\prime})}$ by considering $$\mu_{(\alpha^{\prime},\beta^{\prime})}(U) = \mu_{A^{\prime}}(U \cap A^{\prime}).$$To show that $\mu_{(\alpha^{\prime},\beta^{\prime})}$ is the unique measure of maximal entropy it suffices to show that $A^{\prime}$ is the unique component of maximal entropy of $(\Sigma_{(\alpha^{\prime}, \beta^{\prime})}).$ We claim that $(\Sigma_{(\alpha^{\prime}, \beta^{\prime})})$ has to transitive components, namely $A^{\prime}$ and $B^{\prime} = \left\{\xi_{r_n}^{\infty}, \zeta_{r_n}^{\infty} \right\}$. It is clear that $B^{\prime}$ is a transitive component. Assume that $\left(\Sigma_{(\alpha,\beta)}, \sigma_{(\alpha, \beta)}\right)$ posses another transitive component $C^{\prime}$ with positive topological entropy. This implies there is $M \in \mathbb{M}$ such that $$C^{\prime} \subset \mathop{\bigcup}\limits_{j=M}^{\infty}\rho_{r_n}\circ \ldots \circ \rho_{r_1}(\Sigma_{(\alpha_j, \beta_j)})$$ where $(\alpha_j, \beta_j)$ is the sequence constructed in \cite[Theorem 5.21, Lemma 6.1]{yomero3}. This implies that $(\rho_{r_n}\circ \ldots \circ \rho_{r_1})^{-1}(C^{\prime}) \subset \Sigma_{(\alpha_j, \beta_j)}$ for every $j \geq M$. Furthermore, $$\mathop{\bigcap}\limits_{m=-\infty}^{\infty}\sigma^m((\rho_{r_n}\circ \ldots \circ \rho_{r_1})^{-1}(C^{\prime}))$$ is a transitive component for $\left(\Sigma_{(\alpha_j, \beta_j)},\sigma_{(\alpha_j, \beta_j)}\right)$ for every $j \geq M$. This contradicts that $\left(\Sigma_{(\alpha_j, \beta_j)},\sigma_{(\alpha_j, \beta_j)}\right)$ is coded. Thus, $C^{\prime} = \emptyset$ which concludes the proof. 
\end{proof}

To finish the proof of Theorem \ref{elmerochingon} it is just need to show that $n$-renormalisation boxes cover $D_1$ as we state in the following theorem.

\begin{theorem}
$$\pi^{-1}(D_1) \cap \mathcal{LW}^{\prime} = \mathop{\bigcup}\limits_{n = 0}^{\infty}\mathop{\bigcup}\limits_{(r_1 \ldots r_n) \in \mathbb{Q}\cap(0,1)} \mathop{\bigcup}\limits_{(\alpha, \beta) \in \mathcal{E}}\left(B^n_{(r_1 \ldots r_n)}(\omega, \nu) \cup \rho_{r_n}\circ \ldots \circ \rho_{r_n}(\mathcal{C}^{\prime})\right).$$ \label{asiseacabadecubrir}
\end{theorem}
\begin{proof}
Let $$(0\alpha^{\prime}, 1\beta^{\prime}) \in \mathop{\bigcup}\limits_{n = 0}^{\infty}\mathop{\bigcup}\limits_{(r_1 \ldots r_n) \in \mathbb{Q}\cap(0,1)} \mathop{\bigcup}\limits_{(\alpha, \beta) \in \mathcal{E}}\left(B^n_{(r_1 \ldots r_n)}(\omega, \nu) \cup \rho_{r_n}\circ \ldots \circ \rho_{r_n}(\mathcal{C}^{\prime})\right).$$ Then, there exist $n \geq 0$, $(r_1, \ldots ,r_n) \in (\mathbb{Q}\cap (0,1))^{n}$ and an essential pair $(\alpha, \beta)$ with associated pair $(\omega, \nu)$ such that $(0\alpha^{\prime},1\beta^{\prime}) \in B^n_{(r_1 \ldots r_n)}(\omega, \nu) \cup \rho_{r_n}\circ \ldots \circ \rho_{r_n}(\mathcal{C}^{\prime})$. From Theorem \ref{lemadelaentropia} and the fact that $\dim_{H}(\Sigma_{(\alpha^{\prime}, \beta^{\prime})}) = \frac{h_{top}(\sigma_{(\alpha^{\prime}, \beta^{\prime})})}{\lambda}$ we have that $(0\alpha^{\prime}, 1\beta^{\prime}) \in \pi^{-1}(D_1) \cap \mathcal{LW}^{\prime}$.

Assume now that $(0\alpha^{\prime},1\beta^{\prime}) \in \pi^{-1}(D_1) \cap \mathcal{LW}^{\prime}$. From Theorem \ref{lemacajas1} it suffices to consider $(0\alpha,1\beta) \in \pi^{-1}(D_1 \setminus D_1) \cap \mathcal{LW}^{\prime}$. Then, there is $r \in \mathbb{Q}\cap(0,1)$ such that $(0\alpha,1\beta) \in \left[\xi_{r}^{\infty}, \xi_{r}\zeta_{r}^{\infty}\right]_{\prec} \times \left[\zeta_{r}\xi_{r}^{\infty}, \zeta_{r}^{\infty} \right]_{\prec}$. Moreover, since $\dim_{H}(\Sigma_{(\alpha^{\prime}, \beta^{\prime})}) > 0$ and \cite[Theorem 1.2]{sidorov1} we have that $1\beta \prec \chi(0\alpha)$. From \cite[Corollary 5.15]{yomero3} we have that $(\alpha^{\prime}, \beta^{\prime})$ is renormalisable by $\xi_{r},\zeta_{r}$. Let $r_1 = r$ and consider $(\rho_{r})^{-1}(0\alpha^{\prime}, 1\beta^{\prime})$. Observe that $$(\rho_{r})^{-1}(0\alpha^{\prime}, 1\beta^{\prime}) \in \pi^{-1}(D_2) \cap \mathcal{LW}^{\prime}.$$ Then we have that $$(0\alpha^{\prime},1\beta^{\prime}) \in \mathop{\bigcup}\limits_{n = 0}^{\infty}\mathop{\bigcup}\limits_{(r_1 \ldots r_n) \in \mathbb{Q}\cap(0,1)} \mathop{\bigcup}\limits_{(\alpha, \beta) \in \mathcal{E}}\left(B^n_{(r_1 \ldots r_n)}(\omega, \nu) \cup \rho_{r_n}\circ \ldots \circ \rho_{r_n}(\mathcal{C}^{\prime})\right).$$ We claim that there exist $n \in \mathbb{N}$, $(r_1, \ldots r_n) \in (\mathbb{Q}\cap(0,1))^{n}$ and $(\alpha^{\prime \prime}, \beta^{\prime^\prime})$ such that $$(0\alpha,1\beta) = \rho_{r_n}\circ \ldots \circ \rho_{r_1}(0\alpha^{\prime \prime},1\beta^{\prime \prime}).$$ Suppose that such $n$ does not exist. Then $(\alpha^{\prime},\beta^{\prime})$ is infinitely renormalisable. Then, from \cite[Theorem 2.13]{sidorov1} we have that $h_{top}(\sigma_{(\alpha^{\prime},\beta^{\prime})}) = 0$ which contradicts that $(\alpha^{\prime},\beta^{\prime}) \in \pi^{-1}(D_1) \cap \mathcal{LW}^{\prime}$. Suppose that such $n \in \mathbb{N}$ exist and assume that $n$ is maximal. This gives that there is $(r_1, \ldots r_n) \in (\mathbb{Q}\cap(0,1))^{n}$ such that $(0\alpha,1\beta) = \rho_{r_n}\circ \ldots \circ \rho_{r_1}(0\alpha^{\prime \prime},1\beta^{\prime \prime})$. If $0\alpha^{\prime \prime},1\beta^{\prime \prime} \notin \pi^{-1}(D_2) \cap \mathbb{LW}^{\prime}$. Then there is $r \in \mathbb{Q}\cap(0,1)$ such that $0\alpha^{\prime \prime},1\beta^{\prime \prime} \in \left[\xi_{r}^{\infty}, \xi_{r}\zeta_{r}^{\infty}\right]_{\prec} \times \left[\zeta_{r}\xi_{r}^{\infty}, \zeta_{r}^{\infty} \right]_{\prec}$, which contradicts the maximality of $n$. Thus $$(0\alpha^{\prime},1\beta^{\prime}) \in  \mathop{\bigcup}\limits_{n = 0}^{\infty}\mathop{\bigcup}\limits_{(r_1 \ldots r_n) \in \mathbb{Q}\cap(0,1)} \mathop{\bigcup}\limits_{(\alpha, \beta) \in \mathcal{E}}\left(B^n_{(r_1 \ldots r_n)}(\omega, \nu) \cup \rho_{r_n}\circ \ldots \circ \rho_{r_n}(\mathcal{C}^{\prime})\right).$$
\end{proof}

Then, Theorem \ref{elmerochingon} follows from \cite[Theorem 3.2, Theorem 3.20]{yomero3}, Lemmas \ref{medidapositiva}, \ref{lascajascubren1}, Theorems \ref{ergodicidadintrinseca1}, \ref{ergodicidadintrinseca1renorm} and \ref{asiseacabadecubrir}.

\subsubsection*{Non intrinsically ergodic attractors in $D_1 \setminus D_2$}

We end the section constructing a countable family of non transitive subshifts that are not intrinsically ergodic. Recall that $\omega =  01$ and $\nu_k = 100(10)^k$ for some $k \geq 0$ and consider $(\alpha_k,\beta_k) \in \mathcal{LW}$ given by $$0\alpha = \omega\nu_k^{\infty} \hbox{\rm and }1\beta = \nu_k\omega^{\infty};$$ or $$0\alpha = \overline{\nu_k}{\overline{\omega}}^{\infty} \hbox{\rm and }1\beta = \overline{\omega}{\overline{\nu_k}}^{\infty}.$$ 

\begin{theorem} 
Let $n \in \mathbb{N}$, $(r_1 \ldots r_n) \in (\mathbb{Q}\cap(0,1))^{n}$ and let $(\alpha^{\prime},\beta^{\prime}) \in \mathcal{LW}$ satisfying that $$0\alpha = \rho_{r_n}\circ \ldots \circ \rho_{r_1}(\omega \nu_k^{\infty})$$  and $$ 1\beta = \rho_{r_n}\circ \ldots \circ \rho_{r_1}(\nu_k\omega^{\infty})$$ or $$0\alpha = \rho_{r_n}\circ \ldots \circ \rho_{r_1}(\overline{\nu_k}{\overline{\omega}}^{\infty})$$  and $$1\beta = \rho_{r_n}\circ \ldots \circ \rho_{r_1}(\overline{\omega}{\overline{\nu_k}}^{\infty}).$$ Then $(\Sigma_{(\alpha^{\prime},\beta^{\prime})}, \sigma_{(\alpha^{\prime},\beta^{\prime})})$ is not intrinsically ergodic.\label{laconstruccion}
\end{theorem}
\begin{proof}
From Theorem \ref{canon1} we have that $\left(\Sigma_{(\alpha,\beta)}, \sigma_{(\alpha,\beta)}\right)$ is not intrinsically ergodic since there are two transitive components $A$ and $B$ such that $h_{top}(\sigma\mid_{A})= h_{top}(\sigma\mid_{B})$. Then by Theorem \ref{lemadelaentropia} and Lemma \ref{prepara3renorm} we have that $h_{top}(\sigma\mid_{A^{\prime}})= h_{top}(\sigma\mid_{B^{\prime}}).$ Thus, as we showed in Theorem \ref{canon1} we have there are two $\sigma_{(\alpha^{\prime}, \beta^{\prime})}$ invariant measures with disjoint supports such that $$h_{\mu_{A^{\prime}}} = h_{\mu_{B^{\prime}}} = h_{top}(\sigma_{(\alpha^{\prime},\beta^{\prime})}).$$
\end{proof}

\section*{Acknowledgements}

\noindent The paper extends the content of the author's doctoral dissertation \cite{yomero2} which was fully sponsored by the CONACyT scholarship for Doctoral Students no. 213600. The following research was sponsored by FAPESP 2014/25679-9. The author would like to thank his PhD. supervisor, Nikita Sidorov for his initial suggestion of the problem and his ongoing support. In addition, the author wishes to thank Sofia Trejo Abad for her support and comments during the preparation of the manuscript.

\end{document}